\documentclass[12pt]{article}
\usepackage{amssymb,amsmath,amsthm,amsfonts}
\usepackage{natbib}
\usepackage{setspace}
\usepackage[a4paper,includeheadfoot,margin=1in]{geometry}
\usepackage{xcolor}

\def\b{\boldsymbol}
\def\bs{\boldsymbol}
\def\vec{\mathsf{vec}\,}

\def\argmin{\operatorname*{\mathsf{arg\,min}}}
\def\sup{\operatorname*{\mathsf{sup}}}

\def\tr{\mathsf{trace}\,}
\def\E{\mathbb{E}}
\def\R{\mathbb{R}}

\newcommand{\mn}[1]{{\left\vert\kern-0.25ex\left\vert\kern-0.25ex\left\vert #1
    \right\vert\kern
-0.25ex\right\vert\kern-0.25ex\right\vert}}

\newcommand{\Y}{{\bf Y}}
\newcommand{\bu}{{\bf u}}

\newcommand{\X}{{\bf X}}
\newcommand{\bS}{{\bf S}}
\newcommand{\Z}{{\bf Z}}

\newcommand{\bb}{{\bf b}}

\newcommand{\F}{\mathcal{F}}
\newcommand{\W}{{\bf W}}

\newtheorem{theorem}{Theorem}
\newtheorem*{theorem*}{Theorem}
\newtheorem{definition}{Definition}
\newtheorem*{assumption}{Assumption}
\newtheorem{lemma}{Lemma}
\newtheorem{corollary}{Corollary}
\newtheorem*{remark}{Remark}

\title{Concentration inequalities for high-dimensional linear processes with dependent innovations}
\author{Eduardo F. Mendes\thanks{corresponding author: \texttt{eduardo.mendes@fgv.br}}~\thanks{School of Economics of São Paulo, Fundação Getulio Vargas, Brazil} \and Fellipe Lopes\thanks{School of Applied Mathematics, Fundação Getulio Vargas, Brazil}}

\date{}                     
\begin{document}
\maketitle

\begin{abstract}
    We develop concentration inequalities for the $l_\infty$ norm of a vector linear process on mixingale sequences with sub-Weibull tails. These inequalities make use of the Beveridge-Nelson decomposition, which reduces the problem to concentration for sup-norm of a vector-mixingale or its weighted sum. This inequality is used to obtain a concentration bound for the maximum entrywise norm of the lag-$h$ autocovariance matrices of linear processes. We consider two applications of independent interest: sparse estimation of large-dimensional VAR(p) systems and heterocedasticity and autocorrelation consistent (HAC) high-dimensional covariance estimation.
\end{abstract}

\maketitle
\onehalfspace
\section{Introduction}
In this paper we develop general concentration inequalities that are used in high-dimensional statistical literature. We study a class of high-dimensional time series that can be represented by linear processes with dependent innovations. Specifically, we assume that innovation process is mixingale with sub-Weibull tails. This specification covers a wide range of data-generating processes, such as factor and conditionally heteroskedastic models, as discussed in \citet{wlt2020,bs2021,wbbm2021,mmm2022,asw2023}, among many others.

Linear processes are widely used in time series analysis for their ability to represent a wide range of dependent processes. For example, the Wold decomposition theorem represents stationary nonlinear processes as a linear process with uncorrelated innovations. The Vector Autoregressive Moving Average (VARMA) model, in turn, approximates the time series by a linear process indexed by a finite number of parameters. Typically, innovations are assumed to be independent, allowing for well-understood asymptotic properties \citep{hh1980,ps1992,hl2006}. As an alternative to independence, \cite{wlg2001} shows the weak convergence of partial sums for martingale difference innovations, while \cite{wm2005} establish a central limit theorem and the invariance principle for these processes. \cite{dmp2011} presents maximal inequalities and a functional central limit theorem for innovations in a class of weakly dependent processes. For a comprehensive treatment of multivariate time series, linear representation of nonlinear processes, and illustrative examples, refer to \cite{hl2006,bd2009tsbook} and \cite{tsay2013tsbook}.

High-dimensional statistics deal with the problem when the dimension of the random vector is large, and one is often interested in obtaining concentration bounds for averages that are uniform on the dimension. Typically, these concentrations are derived under independence and either sub-Gaussian or sub-exponential tails as thoroughly discussed in \citet[chap. 2 and 3]{v2018book}, \citet[chap. 3]{mw2019hdbook}, and \cite{hc2021}. In order to account for dependence, many uniform concentration results have been derived for mixing processes such as in \citet{y1994,m1998,fMmPeR2009,mr2010,hs2017,wlt2020,fmm2023} and others. However, mixing is often difficult to verify in practice, favouring alternative forms of dependence, such as weak dependence \citep{weakdependence} and functional dependence \citep{w2005} in \citet{dn2007,ad2011,a2015,zw2017}, among many others. Sub-Weibul tails have recently appeared as a weaker alternative to sub-Gaussian and sub-exponential tails \citep{vgna2020,wlt2020,gss2021,kc2022,zw2022,bk2023}. Sub-Weibull random variables do not have moment generating functions and are often encountered when dealing with products or powers of sub-Gaussian and subexponential random variables. 

\cite{jiang2009} and \citet{cw2018} develop concentration inequalities allowing for both short- and long-range dependence and also heavy-tailed distributions. \cite{jiang2009} focus on a triplex inequality in which the dependence term is characterised by a mixing coefficient, whereas \cite{cw2018} consider a linear process on independent innovations, i.e., dependence comes from the linear weights. Both authors consider a martingale representation of the process, and the tail is a combination of an exponential and second term that can be polynomial or sub-exponential as well. We use a new concentration for martingales in \cite{lv2001,xFiGqL2012} and \citet{xFiGqL2012large}, which yield optimal bounds for sub-Weibull tails, and account for dependence and dimensionality using a mixingale dependence coefficient and Boole inequality, respectively. 

Concentration bounds similar to ours admit a wide range of applications. For example, they are used to derive oracle estimation bounds for VARMA models \citep{wbbm2021}, misspecified VAR($p$) models with $l_1$ penalty \citep{wlt2020, mmm2022}, and $l_1$ penalised Yule-Walker estimation \citep{hll2015,rw2021,wt2021}. Furthermore, these inequalities are essential for the derivation of statistical properties of methods for inference in high-dimensional time series, such as desparsified inference in VAR($p$) models \citep{asw2023} and multiplier bootstrap in high-dimensional time series \citep{kkp2021,asw2023boot}. Finally, the concentration inequality for lag-$h$ autocovariances is also used in the estimation of long-term covariance matrices and spectral density \citep{zw2017,ll2020,bgs2022,fmm2023}. 


\subsection{Notation}
For any vector $\bb = (b_1, ..., b_k)'\in\R^k$ and $p\ge 1$, $|\bb|_p$ is the $\ell_p$ vector-norm with $|\bb|_p = (\sum_{i=1}^k|b_i|^p)^{1/p}$ for $p\in[1,\infty)$ and $|\bb|_\infty = \max_{1\leq i\leq k} |b_i|$. For a random variable $X$, $\|X\|_p = (\E|X|^p)^{1/p}$ for $p\in[1,\infty)$  and $\|X\|_\infty = \inf\{a\in\R:\Pr(|X|\ge a)=0\}$. For an $(m\times n)$-dimensional matrix $\bs{A}$ with elements $a_{ij}$, $\mn{\b A}=\sqrt{\Lambda_{\max}(\b{A'A})}$ is its spectral norm, $\mn{\bs{A}}_1 = \max_{1\le j\le n}\sum_{i=1}^m|a_{ij}|$ and $\mn{\bs{A}}_\infty = \max_{1\le i\le m}\sum_{j=1}^n|a_{ij}|$ are the induced $l_\infty$ and $l_1$ norms, respectively. The maximum entry-wise norm of $\b A$ is $\mn{\bs{A}}_{\max} = |\vec(\b A)|_\infty = \max_{i,j} |a_{ij}|$ and its Frobenius norm is $\mn{\b A}_F = |\vec(\b A)|_2=\sqrt{\tr(\b{A'A})}$. The minimum and maximum eigenvalues of a square matrix $A$ are $\Lambda_{\min}(\b A)$ and $\Lambda_{\max}(\b A)$, respectively. We shall use $c$, $c_1$, $c_2$, ... as generic constants that may change values each time they appear. A constant with a symbolic subscript is used to emphasise the dependence of the value on the subscript. The vector $\bs{e}_i$ is a canonical basis vector of adequate dimension.

\section{Preliminaries}\label{s:basic}
In this section, we introduce the dependence and tail conditions used in the paper, and develop a concentration inequality for the $\ell_\infty$ vector norm (sup-norm) of sums of dependent random variables. 


\subsection{Sub-Weibull random variables}
Sub-Weibull random variables accommodate a wide range of tail behaviour, including variables with heavy tails for which the moment-generating function does not exist, subexponential, and sub-Gaussian. Sub-Weibull random variables are studied in \cite{wlt2020, vgna2020, gss2021, kc2022, zw2022, bk2023}, and it has also appeared in the context of entries of a random matrix in \citet[Condition $C0$]{tv2013}.
\begin{definition}[Sub-Weibull($\alpha$) random variable]\label{d:subweibull}
    Let $\alpha>0$. A sub-Weibull($\alpha$) random variable $X$ satisfies
    $\Pr(|X|>x)\le 2\exp\{-(x/K)^\alpha\}$,for all $x>0$ and some $K>0$.
\end{definition}
There are equivalent definitions in terms of moments and moment generating functions of a $|X|^\alpha$ and Orlicz (quasi-) norms, denoted $\|\cdot\|_{\psi_\alpha} := \inf\{c>0:\E\psi_\alpha(|\cdot|/c)\le 1\}$, with $\psi_\alpha(\cdot) = \exp(x^\alpha)-1$ and $\alpha>0$. In the online supplement, we discuss Orlicz norms and sub-Weibull random variables in detail. An important tail bound that follows after Markov's inequality is
\begin{equation}\label{eq:orlicz-tail}
    \Pr\left(\max_{1\le i\le n}|X_i| > x\right) \le \exp\left(-\frac{x^\alpha}{(c_1\max_{1\le i\le n}\|X_i\|_{\psi_\alpha})^\alpha \log(1+2n)}\right).
\end{equation}

Our approach differs from that of \cite{kc2022,zw2022} and \cite{bk2023} in several key aspects. Unlike the aforementioned authors, who assume independence, we consider a dependent random sequence, which significantly alters the proof methodology. Furthermore, we employ the classical Orlicz norm approach \citep[Sec. 2.2]{vVjW1996book}, in contrast to the generalised Bernstein-Orlicz norm used by the authors. Despite these differences, we demonstrate that, under our assumptions, our rate is nearly optimal, albeit slower than that for sums of independent sub-Weibull random variables.

\subsection{Mixingales}
We characterise the dependence in the process using the moments and conditional moments of the series. This mode of dependence is weak in the sense that it only requires the conditional moments to converge to their marginals in $L_p$ as we have conditioned further in the past. Some classical measures of dependence, such as strong mixing, imply mixingale dependence.
\begin{definition}[Mixingale]
    Let $\{X_t\}$ be a causal stochastic process, and let $\{\mathcal{F}_t\}$ be an increasing sequence of $\sigma$ fields in such a way that $X_t$ is $\mathcal{F}_t$ measurable. The process $\{X_t\}$ is an $L_p$-mixingale process with respect to $\{\mathcal{F}_t\}$ if there exists a decreasing sequence $\{\rho_m\}$ and a constant $c_t$ satisfying $    \left\|\E[X_t|\mathcal{F}_{t-m}]-E[X_t]\right\|_p \le c_t\rho_m$.
\end{definition}

Mixingales can be represented as as a sum of a martingale difference terms and a conditional expectation:
\begin{equation}\label{eq:sumXt}
    \sum_{t=1}^T \left(\X_t-\E[\X_t]\right) = \sum_{i=1}^m \left(\sum_{t=1}^T V_{i,t}\right) + \sum_{t=1}^T \E[\X_t - \E(\X_t)|\F_{t-m}],
\end{equation}
where $V_{i,t} = \E[\X_t|\mathcal{F}_{t-i+1}] - \E[\X_t|\F_{t-i}]$ is a martingale difference process.

\subsection{Concentration inequality}
We present the triplex inequality, based on \citep{jiang2009}, followed by a discussion of the concentration rates.
\begin{theorem}[Concentration for sub-Weibul mixingale processes]\label{t:triplex}
    Let $\{\X_t = (X_{1t},...,X_{nt})'\}$ be a causal stochastic process and $\{\F_t\}$ an increasing sequence of $\sigma$-algebras such that $\X_t$ is $\F_t$ measurable, and write $\bS_k = \sum_{t=1}^k(\X_t-\E[\X_t])$. 
    Suppose that each $\{X_{jt},\F_t\}$ is $L_p$-mixingale with constants $c_{jt}$ and $\{\rho_{jm}\}$, and let $\rho_m = \max_{1\le j\le n}\rho_{jm}$ and $\bar{c}_T = \max_{1\le j\le n}T^{-1}\sum_{t=1}^Tc_{jt}$.  Furthermore, suppose that $\max_{i,t}\|X_{it}\|_{\psi_\alpha} < c_{\psi_\alpha} <\infty$.
    Then, for any natural $m$ and scalar $M>0$: 
    \begin{equation}
        \label{eq:triplex}
        \begin{split}
        \Pr\left(\max_{1\le k\le T}\left|\bS_k\right|_\infty > Tx\right) &\le 2mn\exp\left(-\frac{Tx^2}{8(Mm)^2+2x Mm}\right)\\
        & + 4m\exp\left(-\frac{M^\alpha}{c_1\log(3nT)}\right)+\frac{2^p}{x^p}n\rho_m^p\bar{c}_T,
        \end{split}
    \end{equation}
    where $c_1 := (2c_{\psi_{\alpha}}/\log(1.5))^{\alpha}$.
\end{theorem}


It is natural to ask how tight these bounds are in terms of rate. \citet{lv2001} and \citet{xFiGqL2012large} show that martingales enjoy slower concentration rates. {\color{blue} It happens because to achieve faster rates we would require further restriction on its quadratic variation process.} Let $\{X_t\}$ be a strictly stationary and ergodic martingale difference sequence with $\sup_t\E[e^{|X_t|^\alpha}]<\infty$, then $P(\sum_{i=1}^TX_i>Tx) \ge O(e^{-c T^{\phi_\alpha}})$ where $\phi_\alpha = \alpha/(\alpha+2)$ \citep[Theorem 2.1][]{xFiGqL2012large}. If $\alpha=2$, i.e., $X_i$ has sub-Gaussian tails, the rate is $O(e^{-c T^{1/2}})$, and, similarly, if $X_i$ has sub-exponential tails then the rate is $O(e^{-c T^{1/3}})$. It contrasts with the classical Azuma-Hoeffding inequality, valid for bounded processes, which yields a rate of $O(e^{-c T})$.  

Consider the case where the dependence term has finite memory, i.e., there is some $m^*\ge 1$ such that $\rho_m=0$ for all $m\ge m^*$, suppose $\log n \lesssim T^{\alpha/(\alpha+4)} \wedge T^{\phi_\alpha}\log(T)^{2/(\alpha+2)}$, and take $M=(Tx^2\log(nT))^{1/(\alpha+2)}$. Then the right-hand side of equation \eqref{eq:triplex} is $O\left(e^{-c r_T^{\phi_\alpha}}\right)$ where $r_T = T/\log(nT)^{2/\alpha}$. It means that under finite dependence the rate is nearly optimal, by a factor of $\log(nT)^{1-\phi_\alpha}$. Specifically, in the sub-Gaussian case the rate is $O(e^{-c T^{1/2}/\log(nT)^{1/2}})$, in the subexponential case the rate is $O(e^{-c T^{1/3}/\log(nT)^{2/3}})$, and in the sub-Weibull case with $\alpha=0.5$ the rate is $O(e^{-c T^{1/5}/\log(nT)^{4/5}})$ compared to $O(e^{-c T^{1/5}})$.

If we drop the finite dependence assumption in favour of a \emph{sub-Weibull} decay to the mixingale dependence rate $\rho_m \le e^{m^\gamma/(pc_\rho)}$ for some $\gamma>0$, the convergence rates will change accordingly. Let $m=M^{\alpha/\gamma}\log(3nT)^{1/\gamma}$ and $M = (Tx^2)^{\phi_{\alpha,\gamma}/\alpha}\log(3nT)^{\phi_{\alpha,\gamma}(2+\gamma)/\alpha\gamma}$ with $\phi_{\alpha,\gamma} = \alpha\gamma/(2\alpha+\gamma(2+\alpha))$. Then, if $\log(3nT)\lesssim T^{(2/\phi_{\alpha,\gamma}-(2+\gamma)/\alpha\gamma)^{-1}}$, 
the right hand side of equation \eqref{eq:triplex} is $O(e^{-c T^{\phi_{\alpha,\gamma}} /  \log(nT)^{ 1-\gamma^{-1}({2+\gamma})\phi_{\alpha,\gamma}}})$:
\begin{equation*}\label{eq:triplex_optimal}
\begin{split}
   \Pr\left(\max_{1\le k\le T}\left|\bS_k\right|_\infty > Tx\right)
&\le 2\exp\left(-\frac{(Tx^2)^{\phi_{\alpha,\gamma}}}{16+4x^{\phi_{\alpha,\gamma}}T^{-\frac{1-\phi_{\alpha,\gamma}}{2}}\log(3nT)^{-\frac{\phi_{\alpha,\gamma}}{\alpha}} }\right) \\
&\quad+(4+2^px^{-p}\bar{c}_T)\exp\left(-\frac{(Tx^2)^{\phi_{\alpha,\gamma}}}{c_1\,\log(3nT)^{ 1-\frac{2+\gamma}{\gamma}\phi_{\alpha,\gamma}}}\right),
\end{split}
\end{equation*}
where $c_1 = (2c_{\psi_{\alpha}}/\log(1.5))^{\alpha} \vee 2c_\rho$.

In order to analyse this term, we have to impose rates for both $\alpha$ and $\gamma$. First note that as $\gamma\rightarrow\infty$, $\phi_{\alpha,\infty}=\phi_\alpha$ and $1-\phi_{\alpha,\gamma}(2+\gamma)\gamma^{-1}\rightarrow 1-\phi_\alpha$, recovering the finite dependence case, as expected. Now we consider $\gamma=2$, that is, $\rho_m^p\le e^{-m^2/c_\rho}$. In the sub-Gaussian case, $\phi_{\alpha,\gamma}=\phi_{2,2}=1/3$ and the rate is $O(e^{-cT^{1/3}/\log(nT)^{1/3}})$, that is, we pay a price of $T^{1/6}\log(nT)^{1/3}$ when compared to the optimal rate for martingales and $(T/\log(nT))^{1/6}$ when compared to the rates obtained for the case of high-dimensional, limited dependence. In the subexponential case the rate is $O(e^{-cT^{1/4}/\log(nT)^{1/4}})$ which compares to $O(e^{-cT^{1/3}})$ and $O(e^{-c T^{1/3}/\log(nT)^{2/3}})$ for the martingale and finite dependence in the high-dimensional case, respectively. Increasing the dependence to $\gamma=1$, yields a rate of  $O(e^{-cT^{1/4}/\log(nT)^{1/4}})$ for the sub-Gaussian tail case and $O(e^{-cT^{1/5}/\log(nT)^{2/5}})$ for the sub-exponential tail case. Finally, considering a sub-Weibull tail with parameter $\alpha=0.5$ yields a rate of $O(e^{-cT^{1/7}/\log(nT)^{2/7}})$, compared to $O(e^{-c T^{1/5}/\log(nT)^{4/5}})$ in the case with a high dimension and finite dependence.

In high-dimensional statistics, we are often interested in the situation where $n\rightarrow\infty$ at some rate depending on $T$. In the next corollary we present a useful result in which the rate is delegated to a parameter $\tau$, which can be made dependent on $n$.
\begin{corollary}[Sub-Weibull Concentration]\label{c:bndsw}
    According to the assumptions of Theorem \ref{t:triplex}, suppose that $\rho_m \le e^{-m^\gamma/(pc_\rho)}$ for some $\gamma>0$. Then, for any $\tau>0$ and all $T\ge \log(n)+\tau$:
    \begin{equation}
        \begin{split}
        \Pr\left(\max_{1\le k\le T}\left|\bS_k\right|_\infty > Tx\right) 
        &\le 2c_\rho^{\frac{1}{\gamma}} (\log(n) + \tau)^{\frac{1}{\gamma}}e^{-\tau}+2^p\bar{c}_Tx^{-p}e^{-\tau}\\
        & + 4c_\rho^{\frac{1}{\gamma}}(\log(n) + \tau)^{\frac{1}{\gamma}} e^{\frac{-(x\sqrt{T})^\alpha}{ c_1\log(3nT)(\log(n)+\tau)^{\frac{\alpha}{2}+\frac{\alpha}{\gamma}} } },
        \end{split}
    \end{equation}
    where $c_1:=(8c_{\psi_{\alpha}}c_\rho^{1/\gamma}/\log(1.5))^{\alpha}$.
\end{corollary}
\begin{proof}
    The proof follows by setting $m := c_\rho^{\frac{1}{\gamma}}(\log(n)+\tau)^{\frac{1}{\gamma}}$ and $mM := x\sqrt{T}/4\sqrt{\log(n)+\tau}$.
\end{proof}

The constants that appear in the inequality are not optimised. A simpler bound for high-dimensional vectors, that is, large $n$, follows by setting $\tau=\log(n)$. Let $c_1 = 2(2c_\rho)^{1/\gamma}$, $c_2:=\beta 2^{\frac{7}{2}\alpha+\frac{\alpha}{\gamma}}(c_{\psi_{\alpha}}c_\rho^{1/\gamma}/\log(1.5))^{\alpha}$ and suppose that $T>2\log(n)$ and $T\le n^{\beta-1}/3$ for some $\beta>1$. Then,
\begin{equation*}\label{eq:bndsw-simple}
    \Pr\left(\max_{1\le k\le T}\left|\bS_k\right|_\infty > Tx\right) 
    \le c_1\frac{\log(n)^{1/\gamma}}{n} + \frac{2^p\bar{c}_T}{nx^p} + 2c_1\log(n)^{1/\gamma} e^\frac{-(x\sqrt{T})^\alpha}{c_2\log(n)^{1+\frac{\alpha}{2}+\frac{\alpha}{\gamma}}}.
\end{equation*}

This inequality sheds some light on the rate of increase in $n$ that we can expect so that the right-hand side converges to zero. If $\alpha = 2$, the sub-Gaussian tail case, we have $\log(n) = o(T^{\gamma/(2\gamma+2)})$, in the sub-exponential tail case, $\alpha=1$, $\log(n) = o(T^{\gamma/(3\gamma+2)})$, and in the case with a heavy tail with $\alpha = 0.5$, we have $\log(n) = o(T^{\gamma/(5\gamma+2)})$. In all cases, a strong dependence will guide the rate. If $\gamma$ is very large, the rates are, respectively, close to $o(T^{1/2})$, $o(T^{1/3})$, and $o(T^{1/5})$. This concentration will be used in the next sections to handle more complex stochastic processes. 

\begin{remark}[Extension to other norms]
    Denote $|\cdot|_\psi$ as some norm on $\R^n$ and let $l_n=\sup_{\bb\in\R^n}|\bb|_\psi/|\bb|_\infty$ denote the compatibility constant between this norm and the supremum norm on the same space. Then, $P(|S_T|_\psi>Tx)\le P(|S_T|_\infty>Tx/l_n)$. The compatibility constant $l_n$ will have an effect on convergence rates. Suppose we are in a finite dependence situation with subexponential tails. Then, we have $P(|S_T|_\psi>Tx)\le O(e^{-cT^{1/3}/(l_n\log(nT))^{2/3}})$, which can be very restrictive depending on $l_n$. Effectively, let the $\psi$ norm be the $\ell_p$ norm on $R^n$. Then $l_n=n^{1/p}$ and the new convergence rate is $P(|S_T|_p>Tx)\le O(e^{-cT^{1/3}/(n^{1/p}\log(nT))^{2/3}})$, which is conservative for high-dimensional vectors. 
\end{remark}

\section{Linear processes with dependent innovations}\label{s:lp}
In this section, we extend the concentration inequality in Theorem \ref{t:triplex} to linear stochastic processes with sub-Weibull tails. We first define the multivariate linear process followed by the Beveridge-Nelson (BN) decomposition \citep{bn1981}. 
The latter decomposes a linear process into simpler components that we can analyse using existing techniques and tools \citep{ps1992}.
In particular, we will use Theorem \ref{t:triplex} to bound the first term and Equation \eqref{eq:orlicz-tail} to bound the remaining terms.

Let $\{C_j\}$ denote the $n\times n$ matrices and $\{\X_t\}$ denote a centered stochastic process taking values in $\R^n$. The linear process $\Y_t$ is
\begin{equation}\label{eq:lp}
    \Y_t = \sum_{j=0}^\infty C_j\X_{t-j} = C(L)\X_t,
\end{equation}
where $C(L) = \sum_{j=0}^\infty C_jL^j$ is a lag polynomial and $L$ is the lag operator satisfying $L\X_t  =\X_{t-1}$. 

The tail behaviour of $\Y_t$ is not directly inherited by $\X_t$, unless the conditions in $\{C_j\}$, discussed in Lemma \ref{l:moments}, are satisfied.
\begin{lemma}
    \label{l:moments}
    Let $\{\b a_j\}$ denote a sequence of elements in $\R^n$ each with a finite $L_1$ norm, $\{\Z_t\}$ a sequence of random vectors satisfying $\sup_{|\b a|_1\le 1}\|\b a'\Z_t\|_\psi\le c_t <\infty$ where $\|\cdot\|_\psi$ is a norm and $c_t$s are positive constants, and let $W_t = \sum_{j=0}^\infty \b a_j'Z_{t-j}$. Then, $\|W_t\|_\psi \le \sum_{j=0}^\infty |\b a_j|_1c_{t-j}$, provided $\sum_{j=0}^\infty|\b a_j|_1<\infty$.
\end{lemma}
The BN decomposition represents the matrix polynomial $C(z)$ as
\begin{equation}\label{eq:bn}
    C(z) = C(1) - (1-z) \widetilde{C}(z),
\end{equation}
where $\widetilde{C}(z) = \sum_{j=0}^\infty \widetilde{C}_jz^j$ and $\widetilde{C}_j = \sum_{k=j+1}^\infty C_k$. Then, the linear process is 
\begin{equation}
    \Y_t = C(L)\X_t = C(1)\X_t - \widetilde\X_t + \widetilde\X_{t-1},
\end{equation}
where $\widetilde\X_t = \widetilde{C}(L)\X_t$. 

If $\X_t$ is sub-Weibull($\alpha$) then each component of $\widetilde\X_t$ is also sub-Weibull($\alpha$), provided $\sum_{j=1}^\infty j|\b e_i'C_j|_1 <\infty$. Let $\{\b e_i = (0,...,0,1,0,...,0)', i=1,...,n\}$, with $1$ on the $i^{th}$ element of the vector, denote the canonical basis vectors for $\R^n$. When applying the lemma \ref{l:moments} to $\max_{i\le n}\|\b e_i'\widetilde\X_t\|_{\psi_\alpha}$, we substitute $\b a_j = \b e_i'\widetilde{C}_j$, which means that we require $\sum_{j=1}^\infty j|\b e_i'C_j|_1 <\infty$ for all $1\le i\le n$, and $\sup_{|\bb|\le 1}\|\bb'\X_t\|_\psi < \infty$. 

\begin{theorem}[Concentration inequality for Linear Processes]\label{t:bndlp}
    Let $\{\X_t = (X_{1t},...,X_{nt})' \}$ be a centered sub-Weibull($\alpha$) causal process taking values in $\R^n$, with sub-Weibull constant $c_{\psi_{\alpha}}$, and let $\{\F_t\}$ be an increasing sequence of $\sigma$-algebras such that $\X_t$ is $\F_t$ measurable. Assume that, for each $i=1,...,n$, $\{X_{it},\F_t\}$ is $L_p$-mixingale with positive constants $\{c_{it}\}$ and decreasing sequence $\{\rho_{im}\}$ and write $\bar{c}_T = \max_{1\le i\le n} T^{-1}\sum_{t=1}^Tc_{jt}$ and $\rho_m = \max_{1\le i\le n}\rho_{im}$.

    Write the linear process $\Y_t = C(L)\X_t$, where $\{C_j\}$ is a sequence of square matrices that satisfy $\max_{1\le i\le n}\sum_{j=1}^\infty j|\b{e}_i'C_j|_1\le \tilde{c}_\infty<\infty$ and denote $c_\infty = \mn{C(1)}_\infty$.

    Then, for any $0<a<1$, $T>0$, $M>0$ and $m=1,2,...$, we have 
    \begin{equation}\label{eq:bndlp}
        \begin{split}
            \Pr\left(\max_{1\le k\le T}\left|\sum_{t=1}^k\Y_t\right|_\infty \ge Tx\right)
            &\le 2mn\exp\left(-\frac{T(ax)^2}{8c_\infty^2(Mm)^2+2c_\infty ax Mm}\right)\\
            & + 4m\exp\left(-\frac{M^\alpha}{c_1\log(3nT)}\right)+\frac{(2c_\infty)^p}{(ax)^p}n\rho_m^p\bar{c}_T\\
            & + \exp\left(-\frac{((1-a)Tx)^\alpha}{c_2 \log(1+2n)}\right)
        \end{split}
    \end{equation}
    where $c_1:=(2 c_{\psi_{\alpha}}/\log(1.5))^\alpha$ and $c_2:= (2 c_{\psi_{\alpha}} \tilde{c}_\infty/\log(1.5))^\alpha$
\end{theorem}

The next corollary removes the dependence of the bound on $M$, $m$, and $a$, replacing them with appropriate sequences and constants. 
\begin{corollary}\label{c:bndlp}
    Under the assumptions of Theorem \ref{t:bndlp}, let $\rho_m\le e^{-m^\gamma/(pc_\rho)}$ for some $\gamma>0$. Then, for any $\tau>0$ and $T>\log(n)+\tau$,
    \begin{equation*}
        \begin{split}
            \Pr\left(\max_{1\le k\le T}\left|\sum_{t=1}^k\Y_t\right|_\infty \ge Tx\right)
            &\le c_1(\log(n)+\tau)^{1/\gamma}e^{-\tau} + e^{-\frac{(xT)^{\alpha}}{c_2\log(1+2n)}}+\frac{(4c_\infty)^p\bar{c}_T}{x^p}e^{-\tau}\\
            & + 2c_1(\log(n)+\tau)^{1/\gamma}e^{-\frac{(x\sqrt{T})^\alpha}{c_3\log(3nT)(\log(n)+\tau)^{\frac{\alpha}{2}+\frac{\alpha}{\gamma}}}},
        \end{split}
    \end{equation*}
    where $c_1:=2(c_\rho)^{1/\gamma}$, $c_2:=(4c_{\psi_\alpha}\tilde{c}_\infty/\log(1.5))^\alpha$, and $c_3:= (8c_{\psi_\alpha}c_\infty c_\rho^{1/\gamma}/\log(1.5))^\alpha$.
\end{corollary}
\begin{proof}
    The result follows after replacing $c_\infty mM = ax\sqrt{T}/4\sqrt{\log(n)+\tau}$, $m = (c_\rho(\log(n)+\tau))^{1/\gamma}$, and $a=1/2$. The lower bound on $T$ requires $(\log(n)+\tau)/T<1$.
\end{proof}

Let $\tau = \log(n)$ and assume that $T\le n^{\beta-1}/3$ for some $\beta>1$. If we take $\beta=1+\log(3T)/\log(n)$, we have equality and as long as $\log(T)/\log(n)\not\rightarrow\infty$, as both $n$ and $T$ increase, $\beta$ can be taken sufficiently large. This is often the case in high-dimensional statistics where $n$ is either larger or close to $T$ in rate. A simplified inequality is
\begin{equation}\label{eq:bndlp-simple}
    \begin{split}
        \Pr\left(\max_{1\le k\le T}\left|\sum_{t=1}^k\Y_t\right|_\infty \ge Tx\right) 
        &\le c_1\frac{\log(n)^{1/\gamma}}{n}+ e^{-\frac{(xT)^{\alpha}}{c_2\log(1+2n)}}+\frac{(4c_\infty)^p\bar{c}_T}{n x^p} \\
        &\quad + 2c_1\log(n)^{1/\gamma}e^{-\frac{(x\sqrt{T})^\alpha}{c_3 \log(n)^{1+\frac{\alpha}{2}+\frac{\alpha}{\gamma}}}}
    \end{split},    
\end{equation}
where the constants $c_1 :=  2^{1+1/\gamma}(c_\rho)^{1/\gamma}$,  $c_2:=(4c_{\psi_\alpha}\tilde{c}_\infty/\log(1.5))^\alpha$, and $c_3:=\beta 2^{\frac{7\alpha}{2}+\frac{\alpha}{\gamma}}(c_{\psi_\alpha}c_\infty c_\rho^{1/\gamma}/\log(1.5))^\alpha$. Therefore, the same rates as discussed in the previous section are recovered. 

\begin{remark}
    Some comments on the previous results are in order:
    \begin{enumerate}
        \item In Theorems \ref{t:triplex} and \ref{t:bndlp}, the stochastic process $\{\X_t\}$ is not assumed to be stationary, only the tail and dependence properties hold. For instance, we can have $\X_t$ a heteroskedastic sequence with $\E[\X_t\X_t']=\Sigma_t$ with eigenvalues bounded away from zero and infinity for all $t\in\Z$;
        \item if $X_t$ is not centered, we must work with $\Z_t = \X_t-\E[\X_t]$, in which case $\Y_t = \E[\Y_t] + \sum_{j=0}^{\infty}C_j \Z_t$, where $\E[\Y_t]=\sum_{j=0}^{\infty}C_j \E[\X_t]$. Note that the values of $\E[\X_t]$ can also change for each $t$, so we can incorporate deterministic trends and seasonality into this process. Nevertheless,
        if $\bar{c}_T<\infty$ for all $T$, then $T^{-1/2}{\log(n)^{1/2+1/\alpha+1/\gamma}}\left|{\sum_{t=1}^T\Y_t}\right|_\infty$ is stochastically bounded;
        \item the sequence $\bar{c}_T$ may increase with $T$, in which case limiting arguments used to obtain an expected rate of increase in $n$ as to make probability bound converge to zero will have to accommodate it;
        \item similarly, $c_\infty$ and $\tilde{c}_\infty$ may depend on $n$, in which case $c_2$ in Theorem \ref{t:bndlp} is no longer a constant. To illustrate the effect of this change, let $a=1/2$ and suppose that $c_\infty$ is fixed but $\tilde{c}_\infty \le c_0( \log(1+2n))^{1/\alpha}$: the last term on the right-hand side of equation \eqref{eq:bndlp} changes from $\exp[(Tx)^\alpha/c_2\log(1+2n)]$ to $\exp[(Tx)^\alpha/c_2' (\log(1+2n))^2]$, for two distinct constants $c_2$ and $c_2'$;
    \end{enumerate}
\end{remark}  

\begin{remark}[Martingale difference]
A simpler but interesting case occurs when the process $\{\X_t,\F_{t-1}\}$ is a martingale difference. We take $m=1$ and $\rho_m=0$ in Theorem \ref{t:bndlp}, which is equivalent to taking $\gamma\rightarrow\infty$, which produces rates for $n$ and a tighter concentration bound. 

\begin{corollary}[Concentration for linear processes on martingale differences]\label{c:bndlp_dm}
    Let $\{\X_t = (X_{1t},...,X_{nt})' \}$ be a centred sub-Weibull($\alpha$) causal process taking values in $\R^n$, with sub-Weibull constant $c_{\psi_{\alpha}}$, and let $\{\F_t\}$ be an increasing sequence of $\sigma$-algebras such that $\X_t$ is $\F_t$ measurable, and assume that $\{\X_t,\F_{t-1}\}$ is a martingale difference process.

    Write the linear process $\Y_t = C(L)\X_t$, where $\{C_j\}$ is a sequence of square matrices that satisfy $\max_{1\le i\le n}\sum_{j=1}^\infty j|\b{e}_i'C_j|_1\le \tilde{c}_\infty<\infty$ and denote $c_\infty = \mn{C(1)}_\infty$.

    Then, for any $T>0$ and $M>0$,
    \begin{equation*}
        \begin{split}
            \Pr\left(\max_{1\le k\le T}\left|\sum_{t=1}^k\Y_t\right|_\infty \ge Tx\right)
            &\le 2n\exp\left(-\frac{T(ax)^2}{8c_\infty^2M^2+2c_\infty ax M}\right)\\
            & + 4\exp\left(-\frac{M^\alpha}{c_1\log(3nT)}\right) + \exp\left(-\frac{((1-a)Tx)^\alpha}{c_2 \log(1+2n)}\right)
        \end{split}
    \end{equation*}
    where $c_1:=(2 c_{\psi_{\alpha}}/\log(1.5))^\alpha$ and $c_2:= (2 c_{\psi_{\alpha}} \tilde{c}_\infty/\log(1.5))^\alpha$
\end{corollary}
\begin{proof}
    We take $m=1$ and $\rho_m = 0$ in Theorem \ref{t:bndlp}.
\end{proof}

Following the same arguments in Equation \eqref{eq:bndlp-simple}, an equivalent, simpler, probability bound is
\begin{equation}\label{eq:bndlp-simple_md}
    \begin{split}
        \Pr\left(\left|\sum_{t=1}^T\Y_t\right|_\infty \ge Tx\right)
        &\le \frac{2}{n}+ e^{-\frac{(xT)^{\alpha}}{c_1\log(1+2n)}}+ 4e^{-\frac{(x\sqrt{T})^\alpha}{2^{1+\frac{\alpha}{2}}\beta c_2\log(n)^{1+\frac{\alpha}{2}}}},
    \end{split}
\end{equation}
where $c_1:=(4c_{\psi_\alpha}\tilde{c}_\infty/\log(1.5))^\alpha$, and $c_2:= (8c_{\psi_\alpha}c_\infty/\log(1.5))^\alpha$
\end{remark}

\begin{remark}[Integrated processes]
    In this example we show how integrated processes violate the conditions in Theorem \ref{t:bndlp}.
    Let $\{X_t,\F_{t-1}\}$ be a martingale difference process with identity covariance matrix and $C(L)=I_n$, and suppose that $\{\bb'\X_t\}$ has subexponential tails for all unit vectors $\bb\in\R^n$. Let $\Y_1 = \X_1$ and $(1-L)\Y_t=\X_t$, for $t=2,\cdots,T$. Thus, $\Y_t = \bS_t = \sum_{i=1}^t\X_t$ is a martingale process with respect to $\F_{t-1}$. 
    It follows from the Burkhölder inequality that $\|S_{it}\|_p=O(t^{1/2p})$ and $\|\E[S_{it}|\F_{t-m}]\|_p=O(({t-m})^{1/2p})$ if $m<t$ and $\|\E[S_{it}|\F_{t-m}]\|_1=0$ otherwise, that is, $\rho_m=\max(0,(1-m/t))^{1/2p}$ depends on $t$ and $c_t=t^{1/2p}$. It follows from \citet[Theorem 2.1]{xFiGqL2012large} and \citet[Theorem 3.2]{lv2001} that $P(S_t>x)\lesssim e^{-ct^{-1/3}x^{2/3}}$, and therefore $c_\psi=O(t^{1/3})$ also depends on $t$. 
\end{remark}

\section{Empirical lag-h autocovariance matrices}\label{s:autocov}
In this section, we will examine the concentration properties of the empirical lag $h$ autocovariance matrix of sub-Weibull linear processes under the maximum entry-wise norm. This concentration property is important in the LASSO estimation of large vector auto-regressive models and in the use of the multiplier bootstrap. By studying the concentration of the empirical autocovariance matrix, we can better understand the behaviour of these statistical methods in high-dimensional settings. 

Let $\{\X_t\}$ denote a centred, sub-Weibull, causal stochastic process taking values on $\R^n$, and $\Y_t = C(L)\X_t$, $t=1,\cdots,T$, a dependent sequence of random vectors.  Let
\begin{equation}\label{eq:cov}
    \widehat{\Gamma}_T(h) := \frac{1}{T}\sum_{t=h+1}^T\Y_t\Y_{t-h}'\quad\mbox{and}\quad\Gamma_T(h) := \E[\widehat{\Gamma}_T(h)] = \frac{1}{T}\sum_{t=h+1}^T\E[\Y_t\Y_{t-h}'].
\end{equation}  
Our goal is to find a bound for 
\begin{equation}\label{eq:delta}
    \Delta_T(h) = \mn{\widehat{\Gamma}_T(h) - {\Gamma}_T(h)}_{\max} = \left|\vec(\widehat{\Gamma}_T(h) - {\Gamma}_T(h))\right|_\infty,
\end{equation}
which is the maximum element in absolute value of the matrix.

\begin{theorem}\label{t:acovmatrix}
    Let $\{\X_t = (X_{1t},...,X_{nt})' \}$ be a centred sub-Weibull($\alpha$), causal process taking values in $\R^n$, and let $\{\F_t\}$ be an increasing sequence of $\sigma$-algebras such that $\X_t$ is $\F_t$ measurable. 

    Write $\bs\eta_t(k) = \vec(\X_t\X_{t-k}')$ and the stochastic process $\{\bs\eta_t(k) = (\eta_{1t}(k),...,\eta_{n^2t}(k))'\}$ for $k=0,1,...$. The processes $\{\eta_{it}(k),\F_t\}$ are $L_1$-mixingale with constants $c_{it}$ and decreasing sequences $\rho_{im}$, for each $k=0,1,...$ and $i=1,...,n^2$. Let $\bar{c}_T = \max_{1\le i\le n^2} T^{-1}\sum_{t=1}^Tc_{it}$ and $\rho_m = \max_{1\le i\le n}\rho_{im}\le e^{-m^\gamma/c_\rho}$.
    
    Finally, let the linear process $\Y_t = C(L)\X_t$, where $\{C_j\}$ is a sequence of square matrices, and define the following finite constants:
    \begin{enumerate}
        \item[a. ] $\tilde{c}_{2,\infty}:= \sum_{j=1}^\infty j \mn{C_j}^2$;
        \item[b. ] $c_h := \max_{1\le k\le h}\mn{\sum_{j=0}^\infty C_{j+k}\otimes C_j}_\infty$;
        \item[c. ] $c_\infty:= \max_{1\le k\le h}\mn{\sum_{i=0}^\infty(\sum_{j=i+k}^\infty C_j)\otimes C_i}_\infty$;
        \item[d. ] $\tilde{c}_\infty := \max_{1\le i\le n^2}\sum_{j=1}^\infty j\left|\sum_{k=0}^\infty\bs{e}_i'(C_{j+k}\otimes C_k)\right|_1$.
    \end{enumerate}
    Let $\Delta_T(h)$ be as defined in equations \eqref{eq:cov} - \eqref{eq:delta}. Then, for each $n$ and $T$ that satisfies $T>4\log(n)$ and $3T <n^{\beta-1}$ for some $\beta>1$: 
    \begin{equation}\label{eq:acovmatrix}
        \begin{split}
        \Pr(\Delta_T(h) > x) 
        &\le c_1\frac{\log(n)^{1/\gamma}}{n^2} + \frac{c_2\bar{c}_T}{n^2x} \\
        &+ c_3\log(n)^{1/\gamma}e^{-\frac{(x\sqrt{T})^{\alpha/2}}{c_4\log(n)^{1+\frac{\alpha}{4}+ \frac{\alpha}{2\gamma}}}} + 4e^{-\frac{(xT)^{\alpha/2}}{c_5\log(n)}} 
        \end{split},
    \end{equation}
    where the constants $c_1,...,c_5$ depend only on $\beta$, $c_\rho, c_{\psi_{\alpha/2}}, \tilde{c}_{2,\infty}, c_h, c_\infty, \tilde{c}_\infty$, but not on $n$ or $T$. 
\end{theorem}

Conditions (a) - (d) regulate the persistence of the linear weight matrices $\{C_j\}$ by imposing summation conditions. Conditions (b) - (d) are satisfied by the simpler bound $\max_{1\le i\le n}\sum_{j=1}^\infty j^2|\b{e}_i'C_j|_1 <\infty$. Specifically, applying Hölder's inequality, $c_h\le \max_{j}\mn{C_j}_{\max}\,\max_{1\le i\le n}\sum_{j=1}^\infty|\b{e}_i'C_j|_1$, $c_\infty\le \max_{j}\mn{C_j}_{\max}\,\max_{1\le i\le n}\sum_{j=1}^\infty j|\b{e}_i'C_j|_1$, and $\tilde{c}_\infty \le \max_{j}\mn{C_j}_{\max}\cdot \max_{1\le i\le n}\sum_{j=1}^\infty j^2|\b{e}_i'C_j|_1$.

If $\{\X_t\}$ is centered, uncorrelated in time, and $L_2$-mixingale, one can show that for $k=1,2,...$, $\|\E[X_{it}X_{j,t-k}|\F_{t-m}]\|_1 \le \tilde{c}_{k,t} e^{-m^\gamma/2c_\rho}$ for $\tilde{c}_{k,t} >(c_{it}+c_{it}c_{j,t-k}\rho_m^{3/2}+\|X_{i,t}X_{j,t-k}\|_2\|X_{j,t-k}\|_2)$. Hence, $\{\eta_{it}(k),\F_t\}$ is $L_1$-mixingale with the same ``sub-Weibull'' rate as $\{X_{it},\F_t\}$.  For $k=0$ we directly assume that $\{X_{it}X_{jt},\F_t\}$ is $L_1$ mixingale. \citet{mmm2022} discusses this condition, providing illustrative examples.

\begin{remark}[Martingale difference]
If the process $\X_t,\F_{t-1}$ is a martingale difference process, $\E[X_{it}X_{j,t-k}|\F_{t-m}]=0$ for $k=\pm 1, \pm 2, ...$, which is equivalent to taking $\gamma\rightarrow\infty$. However, the process $\{\widetilde{\bs\eta}_t(0),\F_{t-1}\}$ is not a martingale difference and we still need the $L_1$-mixingale condition in the squared process. After small changes, we obtain a slightly tighter version of Theorem \ref{t:acovmatrix}, but the leading terms remain unchanged.

\end{remark}

\begin{remark}[Accounting for centering]
Unless $\Y_1,..,\Y_T$ are centred, we have to account for centering. If $\E[\Y_t]=\mu_t$, $t=1,..,T$, are known, we only have to demean $\Y_t$ and nothing changes as we may assume $\mu_t=0$ without loss of generality in all derivations. However, $\mu_t$s are typically unknown and must be estimated. Consider the simple case where $\{\X_t\}$ is weakly stationary, so that $\mu_t=\mu$, and define the estimators for the mean $\bar{\Y}^{(1)}_{T-h}=\frac{1}{T-h}\sum_{t=1}^{T-h}\Y_t$ and $\bar{\Y}^{(h)}_{T-h}=\frac{1}{T-h}\sum_{t=h+1}^{T}\Y_t$. The lag-$h$ autocovariance estimator $\widehat{\Gamma}_T(h)$ is
\begin{equation}\label{eq:cov_center}
    \widehat{\Gamma}_T^*(h) := \frac{1}{T-h}\sum_{t=h+1}^T(\Y_t-\bar{\Y}^{(h)}_{T-h})(\Y_{t-h}-\bar{\Y}^{(1)}_{T-h})'.
\end{equation}

It follows directly that
\[
\widehat{\Gamma}_T^*(h) = \frac{1}{T-h}\sum_{t=h+1}^T(\Y_t-\mu)(\Y_{t-h}-\mu)' - (\bar{\Y}_{T-h}^{(h)}-\mu)(\bar{\Y}_{T-h}^{(1)}-\mu)',
\]
meaning that after accounting for appropriate scaling and centering,
\begin{equation}\label{eq:delta_cov_center}
\mn{\widehat{\Gamma}_T^*(h)-\E\widehat{\Gamma}_T^*(h)}_{\max}\le \frac{T}{T-h}\Delta_T(h) + \max_{i=1,h}|\bar{\Y}_{T-h}^{(i)}-\mu|^2_\infty.
\end{equation}
The first term on the right-hand side accounts for the covariances, and the second one for the mean.

For fixed $h$, the first term on the right-hand side of \eqref{eq:delta_cov_center} is bounded as in Theorem \ref{t:acovmatrix}, whereas the second term is bounded using Equation \eqref{eq:bndlp-simple}:
\begin{align*}
\Pr\left(\max_{i=1,h}|\bar{\Y}_{T-h}^{(i)}-\mu|^2_\infty>x \right)
&\le O\left( \log(n)^{1/\gamma} e^{-\frac{(xT)^{\alpha/2}}{c_3 \log(n)^{1+\frac{\alpha}{2}+\frac{\alpha}{\gamma}}}} \right).
\end{align*}
As the second term vanishes faster than the first, the convergence rate does not change.
\end{remark}

\begin{remark}[Accounting for $h$]
The concentration bound holds for each $h$, $T$ and $n$ that satisfies the conditions in Theorem \ref{t:acovmatrix}. We have connected $n$ and $T$ to establish the dimension of $\Y_t$. Now, set $0< l_T = (T-h)/T<1$ and suppose that $T\cdot l_T>4\log(n)$ and $3T\cdot l_T<n^{\beta-1}$, for some $\beta>1$. We obtain the following bound in \eqref{eq:acovmatrix}:
\begin{equation}
    \begin{split}
    \Pr(\Delta_T(h) > x) 
    &\le c_1\frac{\log(n)^{1/\gamma}}{n^2} + \frac{c_2\bar{c}^*_{T-h}l_T}{n^2x} \\
    &+ c_3\log(n)^{1/\gamma}e^{-\frac{(x\sqrt{T/l_T})^{\alpha/2}}{c_4\log(n)^{1+\frac{\alpha}{4}+ \frac{\alpha}{2\gamma}}}} + 4e^{-\frac{(xT)^{\alpha/2}}{c_5\log(n)}} 
    \end{split},
\end{equation}
where $\bar{c}^*_{T-h} = \max_{1\le i\le n^2} (T-h)^{-1}\sum_{t=h+1}^Tc_{it}$. If we take $h=(1-\nu)T$, for $0<\nu<1$, then $l_T=\nu$ and the bound remains unchanged in terms of rate. In other words, we may take $h\propto T$ in Theorem \ref{t:acovmatrix}.
\end{remark}

\section{Applications}\label{s:application}
In this section, we describe two applications of concentration inequalities developed in the paper. The first application is to develop a nonasymptotic oracle bound for the regularized $l_1$ system estimation of the VAR($p$) representations of time series. The second application is a concentration inequality for the maximum entry-wise error of the estimation of the long-run covariance of a linear process. We consider the following restrictions on the data-generating process.
\begin{assumption}[DGP]
The stochastic process $\{\Y_t\}$ taking values on $\R^n$ admits the linear representation:
\[
\Y_t = \sum_{j=1}^\infty C_j\bu_{t-j} = C(L)\bu_t,
\]
where {\color{blue}
    (i) the innovation process $\{\bu_t = (u_{1t},\ldots,u_{nt})'\}$ is causal, centred, uncorrelated, and sub-Weibull($\alpha$), with $\max_{i,t}\|u_{it}\|_{\psi_\alpha}\le c_{\psi_\alpha}$;}
    (ii) for each $i=1,...,n$, $\{u_{it}\}$ is $L_2$-mixingale, $\lim_{T\rightarrow\infty}\bar{c}_T<\infty$ and there exist positive constants $c_\rho$ and $\gamma$ so that $\rho_m\le \exp(-m^\gamma/(2c_\rho))$;  
    (ii) for each $i,j=1,...,n$, $\{u_{it}u_{jt}\}$ is $L_1$-mixingale, with $\lim_{T\rightarrow\infty}\bar{c}_T<\infty$ and there exist positive constants $c_\rho$ and $\gamma$ so that $\rho_m\le \exp(-m^\gamma/(2c_\rho))$;
    (iv) the sequence of matrices $\{C_j\}$ satisfy for all $T$, $n$, and some $r\ge 1$, (a) $\sum_{j=1}^\infty j \mn{C_j}^2 <\infty$; and (b) $\max_{1\le i\le n}\sum_{j=1}^\infty j^{r+1}|\b{e}_i'C_j|_1 <\infty$.  
\end{assumption}

\subsection{LASSO estimation of VAR(p)}
Regularised estimation of high-dimensional time series models have recently been the focus of much research, and two excellent reviews on this topic are \citet{bm2021} and \citet{mmm2023}. Closer to our interest, \citet{sBgM2015} and \citet{aKlC2015} consider LASSO estimation of Gaussian VAR($p$) models, obtaining estimation and prediction bounds. \cite{wlt2020} extends this setting to VAR models on sub-Weibull strictly stationary stochastic processes with $\beta$-mixing dependence. 
\cite{mmm2022} derive nonasymptotic, oracle estimation bounds for linear processes with martingale difference innovations and sub-Weibull tails. 

We consider processes that admit a linear representation with sub-Weibull, mixingale innovations. All previous examples are particular instances of our setup. Note that $\{\Y_t\}$ is not necessarily stationary, but we require the process to be centred at zero. Integrated process do not satisfy our conditions, but processes that are conditionally heteroscedastic are covered by our setup. In effect, the VAR(p) model under consideration approximates the process.
We represent $\Y_t$ using a $VAR(p)$ model
\begin{equation*}\label{eq:varp}
    \Y_t = A_1^*\Y_{t-1}+\cdots+A_p^*\Y_{t-p}+\W_t,
\end{equation*}
where $A_1^*,\ldots,A_p^*$ solve the quadratic program
\[
(A_1^*,\ldots,A_p^*) := \argmin_{(A_1,...,A_p)\in\R^{n\times pn}}\sum_{t=p+1}^T\E\left(\left|\Y_t-\sum_{i=1}^pA_i\Y_{t-i}\right|_2^2\right).
\]
By construction, the error vector $\W_t = \Y_t-\sum_{i=1}^pA_i^*\Y_{t-i}$ ($t=p+1,..,T$) is not correlated with $\Y_{t-1},...,\Y_{t-p}$. 

In high-dimensional vector time series modelling, the number of series $n$ is large compared to the number of observations $T$. It is assumed that the parameters $A_i^*$ are weakly sparse, that is, most entries are (close to) zero; thus, we only estimate a small number of them.
\begin{assumption}[Identification]
The following identification conditions are met. (a) the covariance matrix $\Sigma_T$, with blocks $\Sigma_{T;r,s}=\Gamma_T(r-s)$, has eigenvalues bounded between $0<\sigma_{\Sigma}^2<1$ and $1/\sigma_{\Sigma^2}$ uniformly on $T$; (b) the population parameters satisfy $\sum_{k=1}^p|\vec(A_i^*)|_q^q\le R_q$, for some $0 \le q <1$.
\end{assumption}
The first requirement ensures the least squares program has a unique solution whereas the second that this solution is weakly sparse.

One of the most popular ways for estimating sparse vector autoregressive models is the LASSO (Least Angle Selection and Shrinkage Operator), or the $l_1$ regularised estimator:
\begin{equation*}\label{eq:lasso}
    (\widehat{A}_1,\ldots,\widehat{A}_p) = \argmin_{(A_1,...,A_p)\in\R^{n\times pn}}\sum_{t=p+1}^T\left|\Y_t-\sum_{i=1}^pA_i\Y_{t-i}\right|_2^2 + \lambda \sum_{i=1}^p|\vec(A_i)|_1.
\end{equation*}
The oracle estimation bound below follows as in \citet{mmm2022} Theorem 1, and its adaptation to our case is found in the online supplement. Let $a_\lambda = \min\left(\frac{\lambda}{2(1+\mn{\bs{A}}_{\infty})},\frac{\sigma_\Sigma^{2(1-q)}\lambda^q}{64R_q}\right)$ 
and set
\begin{equation*}\label{eq:problasso}
        \begin{split}
        \pi(a_\lambda) 
        & = c_1\frac{\log(n)^{1/\gamma}}{n^2} + \frac{c_2\bar{c}_T}{n^2a_\lambda} + c_3\log(n)^{1/\gamma}e^{-\frac{(a_\lambda\sqrt{T})^{\alpha/2}}{c_4\log(n)^{1+\frac{\alpha}{4}+ \frac{\alpha}{2\gamma}}}} + 4e^{-\frac{(a_\lambda T)^{\alpha/2}}{c_5\log(n)}} .
        \end{split}
\end{equation*}
Then, with probability $1-2p\pi(a_\lambda)$ and $T>4\log(n)$,
\begin{equation}\label{eq:oracle}
    \sum_{i=1}^p\mn{\widehat{A}_i - A^*_i}_F^2 \le (44+2\lambda)R_q\left(\frac{\lambda}{\sigma_\Sigma^2}\right)^{2-q}.
\end{equation}
Suppose Assumptions DGP and Identification hold with $\sigma_\Sigma^2$ and $R_q$ uniformly bounded for all $n$. Eventually, for $\lambda$ sufficiently small, $a_\lambda = \lambda/2(1+\mn{\bs{A}^*}_\infty)$. Set 
\begin{equation*}\label{eq:lambda}
    \lambda \ge c\,\left(\log\log(n)+\epsilon\right)^{2/\alpha}\log(n)^{2/\alpha+1/\gamma}\sqrt{\frac{\log({n})}{T}},
\end{equation*}
for any $\epsilon>0$ and some constant $c>0$ sufficiently large. Then, the oracle bound will hold with probability at least
\[
1-2p\pi(a_\lambda) = 1-4pe^{-\epsilon}-\frac{2pc_1\log(n)^{1/\gamma}}{n^2} - \frac{(c_2/c)\bar{c}_T\sqrt{T}}{n^2\log(n)^{1/2+2/\alpha+1/\gamma}(\log\log(n)+\epsilon)^{2/\alpha}} .
\]

In practice, $\sigma_\Sigma^2$ and $R_q$ can grow as a function of $n$, in which case the rate of decrease on $\lambda$ would have to accommodate these quantities. We have for $a_\lambda = {\sigma_\Sigma^{2(1-q)}\lambda^q}/{64R_q}$
\[     \lambda^q \ge c\,\frac{\sigma_\Sigma^{2(1-q)}}{R_q}\left(\log\log(n)+\epsilon\right)^{2/\alpha}\log(n)^{2/\alpha+1/\gamma}\sqrt{\frac{\log({n})}{T}}, \] some constant $c$ sufficiently large. However, it is not necessarily a constraint in the rate of $\lambda$, provided that $\lambda R_q^{1/(1-q)}/\sigma_\Sigma^2 = O(1)$.

Consider the simplified setup in which $R_q$ and $\sigma_{\Sigma}^2$ are bounded for all $n$ and $T$. \citet{sBgM2015} obtain the estimation bound in \eqref{eq:oracle} of $O(\log(n)/T)$ with probability $1-O(n^{-c})$. \citet{wlt2020} obtain an oracle bound $O(\log(n)/T)$ with probability $1-O(n^{-c_1}\wedge Te^{-T^{c_2}})$. Finally, in system estimation, \citet{mmm2022} obtain a bound $O(\log(n)^{(4-2q)/\alpha}(\log(n)/T)^{1-q/2})$ with probability $1-O(n^{c_1}T^{c_2}e^{-T^{c_3}}\wedge n^{-c_4})$. In our case, we obtain a bound $O(\log(n)^{(2-q)/\gamma} (\log(n)\log\log(n))^{(4-2q)/\alpha}(\log(n)/T)^{1-q/2})$ with probability $1-O(n^{-2}\wedge T^{1/2}n^{-2}\log(n)^{-c})$. Taking $q=0$ we recover the strong sparsity condition in \citet{sBgM2015} and \citet{wlt2020}, in which case our upper bound is $O(\log(n)^{2/\gamma+4/\alpha}\log\log(n)^{4/\alpha}(\log(n)/T))$ with similar probability. In this case we pay a $O(\log(n)^c)$ price for the relaxed conditions.

\subsection{Long-run covariance matrix}
We present finite sample error bounds for a class of HAC estimators of the long-run covariance $\Omega_T$, which is essential for precise inference in time series. In stationary time series, $\Omega_T$ is related to its spectral density, thus connecting the HAC estimators with the spectral density matrix estimators in statistics and econometrics (see \citet{p2011} and \citet{xw2012}).

Historically, estimators for (long-term) covariance matrices proposed by \citeauthor{nw1987}, \citeauthor{a1991} and \citeauthor{h1992} did not take into account high dimensions. In recent years, researchers have focused on inference for high-dimensional time-series models \citep{zw2017,ll2020,bgs2022,fmm2023}. The most common approach is to assess the characteristics of \citeauthor{nw1987}'s HAC estimator in high dimensions and under certain dependence and tail conditions. In the case of small dimensions, \citet{p2011} shows that the \emph{flat-top} kernels can achieve faster convergence rates. We establish finite sample error bounds for the \citeauthor{p2011}'s flat-top kernel, spectral density matrix estimator under maximum entrywise norm.

We follow Assumption DGP with the added condition that $\{\bu_t\}$ is weakly stationary. Then, the autocovariances $\Gamma(j) = \E[\Y_t\Y_{t-j}] = \Gamma(-j)'$, for $j\ge 0$ and each $t$. The spectral density matrix evaluated at $w$ is
\[
\bs{F}(w) := \frac{1}{2\pi}\sum_{k=-\infty}^\infty \Gamma(k)e^{-ikw},
\]
where $i=\sqrt{-1}$. For $\pi\le w\le \pi$, $\bs{F}(w)$ is positive definite and Hermitian, and $\lim_{T\rightarrow\infty}\Omega_T = 2\pi\bs{F}(0)$. For each $h=0,\pm 1,\ldots$, the sample covariance is $\widehat{\Gamma}_T(h) = \widehat{\Gamma}_T(-h)'$ for $|h|<T$ and $\widehat{\Gamma}_T(h)=0$ otherwise. The empirical spectral matrix estimator is
\[
\widehat{\bs{F}}(w) := \frac{1}{2\pi}\sum_{h=-T+1}^{T-1}\kappa_{g,\epsilon}(h/M_T)\widehat{\Gamma}_T(h),
\]
where $\kappa_{g,\epsilon}(\cdot)$ is a flat-top kernel. The typical flat-top kernel is given by $\kappa_{g,\epsilon}(u)=1$ if $|u|\le\epsilon$ and $\kappa_{g,\epsilon}(u)=g(u)-1$ if $|u|>\epsilon$, where $\epsilon>0$ is a parameter and $g:\R\mapsto [-1,1]$ is a symmetric function, continuous at all but a finite number of points, satisfying $g(\epsilon)=1$ and $\int_\R g^2(u)du <\infty$.

It follows after some algebra (see online supplement) that in a set with large probability and for any $w\in[0,2\pi]$,
\begin{equation*}\label{eq:F_error_bound}
    \mn{\widehat{\bs{F}}(w) -\bs{F}(w)}_{\max}\lesssim  \frac{2c_r}{\pi\epsilon^r M_T^r} + \frac{c_0}{\pi T} +\frac{H_T}{\sqrt{T}}\log(n)^{1+\frac{2}{\alpha}+\frac{1}{\gamma}}(\log\log(n)+\log(T)+\tau)^{2/\alpha}.
\end{equation*}
Naturally, we require the right-hand side to converge to zero. Therefore, $H_T$ and $M_T$ cannot grow too fast. Such restrictions are expected and are also observed in \citet{ll2020} and \citet{bgs2022}.
\begin{remark}
    Let $\Y_t = f(\Z_t;\theta_0)$ where the population parameter $\theta_0$ is estimated by $\widehat\theta_T$. In this case, we have access to $\widehat\Y_t = f(\Z_t;\widehat\theta_T)$ and calculate
    \[
    \widetilde\Gamma_T(h) = \frac{1}{T}\sum_{t=h+1}^T\widehat\Y_t\widehat\Y_{t-h}'\quad\mbox{and}\quad \widetilde\Gamma_T(-h) = \widetilde\Gamma_T(h),\quad h=0,1\ldots,T-1.
    \]
    The spectral density estimator is
    \[
    \widetilde{\bs{F}}(w) = \sum_{h=-T+1}^{T-1}\kappa_{g,\epsilon}(h/M_T)\widetilde\Gamma_T(h).
    \]
    If the estimation error $\max_{i\le n}\sum_{t=1}^T\left(Y_{i,t}-\widehat Y_{i,t}\right)^2< T\delta_{n,t}^2$ with high probability, we obtain 
    \[
    \begin{split}
        \mn{\widetilde{\bs{F}}(w)-\bs{F}(w)}_{\max}&\le \delta_{n,T} + \frac{2c_r}{\pi\epsilon^r M_T^r} + \frac{c_0}{\pi T}\\
        &+\frac{H_T}{\sqrt{T}}\log(n)^{1+\frac{2}{\alpha}+\frac{1}{\gamma}}(\log\log(n)+\log(T)+\tau)^{2/\alpha},
    \end{split}
    \]
    with large probability as well.
\end{remark}

\section{Discussion}\label{s:discussion}

In this paper, we explore the concentration bounds for the supremum norm of averages of vector-valued linear processes. We demonstrate that, when the weights are summable, the rates obtained for sums of mixingales can be extended to sums of linear processes. This result is noteworthy for two reasons: it does not require stationarity, and the constants depending on the weights or the mixingale term can increase with the dimension $n$ of the process or the sample size $T$. Additionally, the mixingale condition is a well-known condition used in the derivation of asymptotic properties of time series estimators, as well as in the analysis of properties of nonlinear models. \citet[Section 3]{mmm2022} provides some examples of processes that meet the mixingale and tail conditions.

This article generalises the concentration results of \citet{wlt2020}, which requires $\Y_t$ to be $\beta$ mixed, \citet{mmm2022}, which requires $Y_t$ to be an approximately VAR($p$) process, and \citet{bs2021}, which assumes i.i.d. innovations. \citet{mmm2022} was extended to encompass misspecified linear models, where the mean is not an approximately sparse VAR model. \citet{bs2021} showed that factor models and generalised factor models can be represented by a linear process. This is also applicable to processes with stochastic variance, as in \citet[Section 3]{mmm2022}.

\citet{kc2022} demonstrate that the concentration inequalities developed in this paper can be used in a variety of statistical contexts, including the derivation of estimation bounds for the maximum $k$-sub-matrix operator norm and the restricted isometry property. Additionally, the restricted eigenvalue condition and restricted eigenvalue condition for linear time series models were shown to hold. 

We use our concentration inequalities to obtain estimation bounds for weakly sparse $VAR(p)$ autoregressions and error bounds for high-dimensional HAC estimators with flat-top kernels. We extend the work of \citet{sBgM2015, aKlC2015,wlt2020} and \citet{mmm2022} by considering a more general data-generating process. Additionally, we provide estimation error bounds for the HAC, flat-top kernel estimator from \citet{p2011} in a high-dimensional setting. This concentration bound has been used in high-dimensional time series inference, as demonstrated in \citet{zw2017, ll2020, bgs2022} and \cite{fmm2023}.

\section*{Acknowledgements}
We thank professors Marcelo Fernandes and Luiz Max Carvalho, and two anonymous referees for their insightful comments. The first author acknowledge financial support from the São Paulo Research Foundation (FAPESP), grant 2023/01728-0
\par

\bibliographystyle{chicago}      
\bibliography{mybib}             

\begin{thebibliography}{}

\bibitem[\protect\citeauthoryear{Adamczak}{Adamczak}{2015}]{a2015}
Adamczak, R. (2015).
\newblock {A note on the Hanson-Wright inequality for random vectors with
  dependencies}.
\newblock {\em Electronic Communications in Probability\/}~{\em 20\/}(none), 1
  -- 13.

\bibitem[\protect\citeauthoryear{Adamek, Smeekes, and Wilms}{Adamek
  et~al.}{2023a}]{asw2023}
Adamek, R., S.~Smeekes, and I.~Wilms (2023a).
\newblock Lasso inference for high-dimensional time series.
\newblock {\em Journal of Econometrics\/}~{\em 235\/}(2), 1114--1143.

\bibitem[\protect\citeauthoryear{Adamek, Smeekes, and Wilms}{Adamek
  et~al.}{2023b}]{asw2023boot}
Adamek, R., S.~Smeekes, and I.~Wilms (2023b).
\newblock Sparse high-dimensional vector autoregressive bootstrap.

\bibitem[\protect\citeauthoryear{Alquier and Doukhan}{Alquier and
  Doukhan}{2011}]{ad2011}
Alquier, P. and P.~Doukhan (2011).
\newblock Sparsity considerations for dependent variables.
\newblock {\em Electronic Journal of Statistics\/}~{\em 5}, 750--774.

\bibitem[\protect\citeauthoryear{Andrews}{Andrews}{1991}]{a1991}
Andrews, D.~W. (1991).
\newblock Heteroskedasticity and autocorrelation consistent covariance matrix
  estimation.
\newblock {\em Econometrica: Journal of the Econometric Society\/}~{\em
  59\/}(3), 817--858.

\bibitem[\protect\citeauthoryear{Babii, Ghysels, and Striaukas}{Babii
  et~al.}{2022}]{bgs2022}
Babii, A., E.~Ghysels, and J.~Striaukas (2022).
\newblock High-dimensional grander causality tests with an application to vix
  and news.
\newblock {\em Journal of Financial Econometrics\/}.

\bibitem[\protect\citeauthoryear{Basu and Matteson}{Basu and
  Matteson}{2021}]{bm2021}
Basu, S. and D.~S. Matteson (2021).
\newblock A survey of estimation methods for sparse high-dimensional time
  series models.

\bibitem[\protect\citeauthoryear{Basu and Michailidis}{Basu and
  Michailidis}{2015}]{sBgM2015}
Basu, S. and G.~Michailidis (2015).
\newblock Regularized estimation in sparse high-dimensional time series models.
\newblock {\em The Annals of Statistics\/}~{\em 43}, 1535--1567.

\bibitem[\protect\citeauthoryear{Beveridge and Nelson}{Beveridge and
  Nelson}{1981}]{bn1981}
Beveridge, S. and C.~R. Nelson (1981).
\newblock A new approach to decomposition of economic time series into
  permanent and transitory components with particular attention to measurement
  of the ‘business cycle’.
\newblock {\em Journal of Monetary economics\/}~{\em 7\/}(2), 151--174.

\bibitem[\protect\citeauthoryear{Bong and Kuchibhotla}{Bong and
  Kuchibhotla}{2023}]{bk2023}
Bong, H. and A.~K. Kuchibhotla (2023).
\newblock Tight concentration inequality for sub-weibull random variables with
  generalized bernstien orlicz norm.

\bibitem[\protect\citeauthoryear{Bours and Steland}{Bours and
  Steland}{2021}]{bs2021}
Bours, M. and A.~Steland (2021).
\newblock Large-sample approximations and change testing for high-dimensional
  covariance matrices of multivariate linear time series and factor models.
\newblock {\em Scandinavian Journal of Statistics\/}~{\em 48\/}(2), 610--654.

\bibitem[\protect\citeauthoryear{Brockwell and Davis}{Brockwell and
  Davis}{2009}]{bd2009tsbook}
Brockwell, P.~J. and R.~A. Davis (2009).
\newblock {\em Time series: theory and methods}.
\newblock New York: Springer science \& business media.

\bibitem[\protect\citeauthoryear{Chen and Wu}{Chen and Wu}{2018}]{cw2018}
Chen, L. and W.~B. Wu (2018).
\newblock Concentration inequalities for empirical processes of linear time
  series.
\newblock {\em J. Mach. Learn. Res.\/}~{\em 18}, 231--1.

\bibitem[\protect\citeauthoryear{Davidson}{Davidson}{1994}]{jD1994}
Davidson, J. (1994).
\newblock {\em Stochastic Limit Theory}.
\newblock Oxford: Oxford University Press.

\bibitem[\protect\citeauthoryear{Dedecker, Doukhan, Lang, Le\'{o}n, Louhichi,
  and Prieur}{Dedecker et~al.}{2007}]{weakdependence}
Dedecker, J., P.~Doukhan, G.~Lang, J.~Le\'{o}n, S.~Louhichi, and C.~Prieur
  (2007).
\newblock {\em Weak dependence with examples and applications}.
\newblock Springer.

\bibitem[\protect\citeauthoryear{Dedecker, Merlev{\`e}de, and
  Peligrad}{Dedecker et~al.}{2011}]{dmp2011}
Dedecker, J., F.~Merlev{\`e}de, and M.~Peligrad (2011).
\newblock Invariance principles for linear processes with application to
  isotonic regression.
\newblock {\em Bernoulli\/}~{\em 17\/}(1), 88--113.

\bibitem[\protect\citeauthoryear{Doukhan and Neumann}{Doukhan and
  Neumann}{2007}]{dn2007}
Doukhan, P. and M.~H. Neumann (2007).
\newblock Probability and moment inequalities for sums of weakly dependent
  random variables, with applications.
\newblock {\em Stochastic Processes and their Applications\/}~{\em 117\/}(7),
  878--903.

\bibitem[\protect\citeauthoryear{Fan, Masini, and Medeiros}{Fan
  et~al.}{2023}]{fmm2023}
Fan, J., R.~Masini, and M.~C. Medeiros (2023).
\newblock Bridging factor and sparse models.
\newblock {\em Annals of Statistics\/}~{\em 51\/}(4), 1692--1717.

\bibitem[\protect\citeauthoryear{Fan, Grama, and Liu}{Fan
  et~al.}{2012}]{xFiGqL2012}
Fan, X., I.~Grama, and Q.~Liu (2012).
\newblock Hoeffding’s inequality for supermartingales.
\newblock {\em Stochastic Processes and their Applications\/}~{\em 122\/}(10),
  3545--3559.

\bibitem[\protect\citeauthoryear{Fan, Grama, Liu, et~al.}{Fan
  et~al.}{2012}]{xFiGqL2012large}
Fan, X., I.~Grama, Q.~Liu, et~al. (2012).
\newblock Large deviation exponential inequalities for supermartingales.
\newblock {\em Electronic Communications in Probability\/}~{\em 17}.

\bibitem[\protect\citeauthoryear{G{\"o}tze, Sambale, and Sinulis}{G{\"o}tze
  et~al.}{2021}]{gss2021}
G{\"o}tze, F., H.~Sambale, and A.~Sinulis (2021).
\newblock Concentration inequalities for polynomials in
  $\alpha$-sub-exponential random variables.
\newblock {\em Electronic Journal of Probability\/}~{\em 26}, 1--22.

\bibitem[\protect\citeauthoryear{Hall and Heyde}{Hall and Heyde}{1980}]{hh1980}
Hall, P. and C.~C. Heyde (1980).
\newblock {\em Martingale limit theory and its application}.
\newblock New York: Academic press.

\bibitem[\protect\citeauthoryear{Han, Lu, and Liu}{Han et~al.}{2015}]{hll2015}
Han, F., H.~Lu, and H.~Liu (2015).
\newblock A direct estimation of high dimensional stationary vector
  autoregressions.
\newblock {\em The Journal of Machine Learning Research\/}~{\em 16\/}(1),
  3115--3150.

\bibitem[\protect\citeauthoryear{Hang and Steinwart}{Hang and
  Steinwart}{2017}]{hs2017}
Hang, H. and I.~Steinwart (2017).
\newblock A bernstein-type inequality for some mixing processes and dynamical
  systems with an application to learning.
\newblock {\em Annals of Statistics\/}~{\em 45\/}(2), 708--743.

\bibitem[\protect\citeauthoryear{Hansen}{Hansen}{1992}]{h1992}
Hansen, B.~E. (1992).
\newblock Consistent covariance matrix estimation for dependent heterogeneous
  processes.
\newblock {\em Econometrica: Journal of the Econometric Society\/}~{\em
  60\/}(4), 967--972.

\bibitem[\protect\citeauthoryear{Jiang}{Jiang}{2009}]{jiang2009}
Jiang, W. (2009).
\newblock On uniform deviations of general empirical risks with unboundedness,
  dependence and high dimensionality.
\newblock {\em Journal of Machine Learning Research\/}~{\em 10}, 977--996.

\bibitem[\protect\citeauthoryear{Kock and Callot}{Kock and
  Callot}{2015}]{aKlC2015}
Kock, A. and L.~Callot (2015).
\newblock Oracle inequalities for high dimensional vector autoregressions.
\newblock {\em Journal of Econometrics\/}~{\em 186}, 325--344.

\bibitem[\protect\citeauthoryear{Krampe, Kreiss, and Paparoditis}{Krampe
  et~al.}{2021}]{kkp2021}
Krampe, J., J.-P. Kreiss, and E.~Paparoditis (2021).
\newblock Bootstrap based inference for sparse high-dimensional time series
  models.
\newblock {\em Bernoulli\/}~{\em 27\/}(3), 1441--1466.

\bibitem[\protect\citeauthoryear{Kuchibhotla and Chakrabortty}{Kuchibhotla and
  Chakrabortty}{2022}]{kc2022}
Kuchibhotla, A.~K. and A.~Chakrabortty (2022).
\newblock {Moving beyond sub-Gaussianity in high-dimensional statistics:
  applications in covariance estimation and linear regression}.
\newblock {\em Information and Inference: A Journal of the IMA\/}~{\em
  11\/}(4), 1398--1456.

\bibitem[\protect\citeauthoryear{Lesigne and Volný}{Lesigne and
  Volný}{2001}]{lv2001}
Lesigne, E. and D.~Volný (2001).
\newblock Large deviations for martingales.
\newblock {\em Stochastic Process and their Applications\/}~{\em 96}, 143 --
  159.

\bibitem[\protect\citeauthoryear{Li and Liao}{Li and Liao}{2020}]{ll2020}
Li, J. and Z.~Liao (2020).
\newblock Uniform nonparametric inference for time series.
\newblock {\em Journal of Econometrics\/}~{\em 219\/}(1), 38--51.

\bibitem[\protect\citeauthoryear{Lütkepohl}{Lütkepohl}{2006}]{hl2006}
Lütkepohl, H. (2006).
\newblock {\em New Introduction to Multiple Time Series Analysis}.
\newblock Berlin: Springer-Verlag.

\bibitem[\protect\citeauthoryear{Marton}{Marton}{1998}]{m1998}
Marton, K. (1998).
\newblock Measure concentration for a class of random processes.
\newblock {\em Probability Theory and Related Fields\/}~{\em 110}, 427--439.

\bibitem[\protect\citeauthoryear{Masini, Medeiros, and Mendes}{Masini
  et~al.}{2022}]{mmm2022}
Masini, R.~P., M.~C. Medeiros, and E.~F. Mendes (2022).
\newblock Regularized estimation of high-dimensional vector autoregressions
  with weakly dependent innovations.
\newblock {\em Journal of Time Series Analysis\/}~{\em 43}, 532--557.

\bibitem[\protect\citeauthoryear{Masini, Medeiros, and Mendes}{Masini
  et~al.}{2023}]{mmm2023}
Masini, R.~P., M.~C. Medeiros, and E.~F. Mendes (2023).
\newblock Machine learning advances for time series forecasting.
\newblock {\em Journal of Economic Surveys\/}~{\em 37\/}(1), 76--111.

\bibitem[\protect\citeauthoryear{Merlev{è}de, Peligrad, and Rio}{Merlev{è}de
  et~al.}{2009}]{fMmPeR2009}
Merlev{è}de, F., M.~Peligrad, and E.~Rio (2009).
\newblock Bernstein inequality and moderate deviations under strong mixing
  conditions.
\newblock In C.~Houdr{é}, V.~Koltchinskii, D.~Mason, and M.~Peligrad (Eds.),
  {\em High Dimensional Probability V: The Luminy Volume}, pp.\  273--292.
  Institute of Mathematical Statistics.

\bibitem[\protect\citeauthoryear{Mohri and Rostamizadeh}{Mohri and
  Rostamizadeh}{2010}]{mr2010}
Mohri, M. and A.~Rostamizadeh (2010).
\newblock Stability bounds for stationary {$\phi$-mixing} and {$\beta$-mixing}
  processes.
\newblock {\em Journal of Machine Learning Research\/}~{\em 11\/}(2).

\bibitem[\protect\citeauthoryear{Newey and West}{Newey and West}{1987}]{nw1987}
Newey, W.~K. and K.~D. West (1987).
\newblock A simple, positive semi-definite, heteroskedasticity and
  autocorrelation consistent covariance matrix.
\newblock {\em Econometrica\/}~{\em 55\/}(3), 703--708.

\bibitem[\protect\citeauthoryear{Phillips and Solo}{Phillips and
  Solo}{1992}]{ps1992}
Phillips, P.~C. and V.~Solo (1992).
\newblock Asymptotics for linear processes.
\newblock {\em The Annals of Statistics\/}~{\em 20\/}(2), 971--1001.

\bibitem[\protect\citeauthoryear{Politis}{Politis}{2011}]{p2011}
Politis, D.~N. (2011).
\newblock Higher-order accurate, positive semidefinite estimation of
  large-sample covariance and spectral density matrices.
\newblock {\em Econometric Theory\/}~{\em 27\/}(4), 703--744.

\bibitem[\protect\citeauthoryear{Reuvers and Wijler}{Reuvers and
  Wijler}{2024}]{rw2021}
Reuvers, H. and E.~Wijler (2024).
\newblock Sparse generalized yule--walker estimation for large spatio-temporal
  autoregressions with an application to no2 satellite data.
\newblock {\em Journal of Econometrics\/}~{\em 239\/}(1).

\bibitem[\protect\citeauthoryear{Tao and Vu}{Tao and Vu}{2013}]{tv2013}
Tao, T. and V.~Vu (2013).
\newblock Random matrices: Sharp concentration of eigenvalues.
\newblock {\em Random Matrices: Theory and Applications\/}~{\em 2\/}(03),
  1350007.

\bibitem[\protect\citeauthoryear{Tsay}{Tsay}{2013}]{tsay2013tsbook}
Tsay, R.~S. (2013).
\newblock {\em Multivariate time series analysis: with R and financial
  applications}.
\newblock Hoboken, New Jersey: John Wiley \& Sons.

\bibitem[\protect\citeauthoryear{van~der Vaart and Wellner}{van~der Vaart and
  Wellner}{1996}]{vVjW1996book}
van~der Vaart, A. and J.~Wellner (1996).
\newblock {\em Weak Convergence and Empirical Processes: With Applications to
  Statistics}.
\newblock Springer.

\bibitem[\protect\citeauthoryear{Vershynin}{Vershynin}{2018}]{v2018book}
Vershynin, R. (2018).
\newblock {\em High-dimensional probability: An introduction with applications
  in data science}, Volume~47.
\newblock Cambridge University Press.

\bibitem[\protect\citeauthoryear{Vladimirova, Girard, Nguyen, and
  Arbel}{Vladimirova et~al.}{2020}]{vgna2020}
Vladimirova, M., S.~Girard, H.~Nguyen, and J.~Arbel (2020).
\newblock Sub-weibull distributions: Generalizing sub-gaussian and
  sub-exponential properties to heavier tailed distributions.
\newblock {\em Stat\/}~{\em 9\/}(1), e318.

\bibitem[\protect\citeauthoryear{Wainwright}{Wainwright}{2019}]{mw2019hdbook}
Wainwright, M.~J. (2019).
\newblock {\em High-dimensional statistics: A non-asymptotic viewpoint},
  Volume~48.
\newblock Cambridge university press.

\bibitem[\protect\citeauthoryear{Wang and Tsay}{Wang and Tsay}{2023}]{wt2021}
Wang, D. and R.~S. Tsay (2023).
\newblock Rate-optimal robust estimation of high-dimensional vector
  autoregressive models.
\newblock {\em Annals of Statistics\/}~{\em 51\/}(2), 846--877.

\bibitem[\protect\citeauthoryear{Wang, Lin, and Gulati}{Wang
  et~al.}{2001}]{wlg2001}
Wang, Q., Y.-X. Lin, and C.~M. Gulati (2001).
\newblock Asymptotics for moving average processes with dependent innovations.
\newblock {\em Statistics \& probability letters\/}~{\em 54\/}(4), 347--356.

\bibitem[\protect\citeauthoryear{Wilms, Basu, Bien, and Matteson}{Wilms
  et~al.}{2021}]{wbbm2021}
Wilms, I., S.~Basu, J.~Bien, and D.~S. Matteson (2021).
\newblock Sparse identification and estimation of large-scale vector
  autoregressive moving averages.
\newblock {\em Journal of the American Statistical Association\/}~{\em
  118\/}(541), 571--582.

\bibitem[\protect\citeauthoryear{Wong, Li, and Tewari}{Wong
  et~al.}{2020}]{wlt2020}
Wong, K.~C., Z.~Li, and A.~Tewari (2020).
\newblock Lasso guarantees for $\beta$-mixing heavy-tailed time series.
\newblock {\em The Annals of Statistics\/}~{\em 48\/}(2), 1124--1142.

\bibitem[\protect\citeauthoryear{Wu}{Wu}{2005}]{w2005}
Wu, W.~B. (2005).
\newblock Nonlinear system theory: Another look at dependence.
\newblock {\em Proceedings of the National Academy of Sciences\/}~{\em
  102\/}(40), 14150--14154.

\bibitem[\protect\citeauthoryear{Wu and Min}{Wu and Min}{2005}]{wm2005}
Wu, W.~B. and W.~Min (2005).
\newblock On linear processes with dependent innovations.
\newblock {\em Stochastic Processes and Their Applications\/}~{\em 115\/}(6),
  939--958.

\bibitem[\protect\citeauthoryear{Xiao and Wu}{Xiao and Wu}{2012}]{xw2012}
Xiao, H. and W.~B. Wu (2012).
\newblock Covariance matrix estimation for stationary time series.
\newblock {\em The Annals of Statistics\/}~{\em 40\/}(1), 466--493.

\bibitem[\protect\citeauthoryear{Yu}{Yu}{1994}]{y1994}
Yu, B. (1994, January).
\newblock Rates of convergence for empirical processes of stationary mixing
  sequences.
\newblock {\em The Annals of Probability\/}~{\em 22\/}(1), 94--116.

\bibitem[\protect\citeauthoryear{Zhang and Wu}{Zhang and Wu}{2017}]{zw2017}
Zhang, D. and W.~B. Wu (2017).
\newblock Gaussian approximation for high-dimensional time series.
\newblock {\em The Annals of Statistics\/}~{\em 45\/}(5), 1895--1919.

\bibitem[\protect\citeauthoryear{Zhang and Chen}{Zhang and Chen}{2021}]{hc2021}
Zhang, H. and S.~Chen (2021).
\newblock Concentration inequalities for statistical inference.
\newblock {\em Communications in Mathematical Research\/}~{\em 37\/}(1), 1--85.

\bibitem[\protect\citeauthoryear{Zhang and Wei}{Zhang and Wei}{2022}]{zw2022}
Zhang, H. and H.~Wei (2022).
\newblock Sharper sub-weibull concentrations.
\newblock {\em Mathematics\/}~{\em 10\/}(13), 2252.

\end{thebibliography}
\newpage
\appendix
\setcounter{section}{0}
\setcounter{equation}{0}
\def\theequation{S\arabic{section}.\arabic{equation}}
\def\thesection{S\arabic{section}}

\section{Introduction}
\
In this additional material, we discuss the features of the innovation process and demonstrate concentration results that are of independent value, which are used in the proof of the triplex inequality. We also present the results of Sections 3 and 4 of the main paper, along with their respective proofs.
\par

\section{Preliminaries}

In this section, we explore the features of the innovation process in more detail. We analyse the properties of the Orlicz quasi-norm for sub-Weibull random variables, demonstrate a maximal inequality, and use this to derive a concentration inequality for the maximum of $n$ sub-Weibull random variables. We then delve deeper into the mixingale condition, the telescopic expansion, and bound the error term from a finite martingale approximation. Finally, we prove a concentration inequality for sub-Weibul martingales and the triplex inequality in Theorem 1.

\subsection{Sub-Weibull random variables}

Random variables of the sub-Weibull class can exhibit a wide range of tail behaviour, including those with heavy tails for which the moment-generating function does not exist. This allows us to broaden our analysis to include sub-Gaussian and subexponential random variables, which are commonly used in high-dimensional literature. Research on sub-Weibull random variables can be found in \cite{wlt2020, vgna2020, gss2021, zw2022, kc2022}, and it has also been used in the context of entries of a random matrix in \citet[Condition $C0$]{tv2013}.

\begin{definition}[Sub-Weibull($\alpha$) random variable]\label{d:subweibull}
    Let $\alpha>0$. A sub-Weibull($\alpha$) random variable $X$ satisfies
    \[
    \Pr(|X|>x)\le 2\exp\{-(x/K)^\alpha\},    
    \]
    for all $x>0$ and some $K>0$.
\end{definition}
There are equivalent definitions in terms of moments and moment generating functions of a $|X|^\alpha$ and Orlicz (quasi-) norms. Let $\psi_\alpha(\cdot) = \exp(x^\alpha)-1$ for $\alpha>0$ and denote the (quasi-)norm 
\[
    \|\cdot\|_{\psi_\alpha} := \inf\{c>0:\E\psi_\alpha(|\cdot|/c)\le 1\}.
\]
It follows that $K=\|X\|_{\psi_\alpha}$ in Definition \ref{d:subweibull}. This is equivalent to requiring that $\sup_{p>0}p^{-1/\alpha}\|X\|_p < K$ for some constant $K$. In particular, the following equivalence will be useful \citep[Lemma A.2]{gss2021}. Let $d_\alpha := (e\alpha)^{1/\alpha}/2$ and $D_\alpha := (2e(\alpha\wedge 1))^{1/\alpha}$:
\begin{equation}
    d_\alpha \sup_{p>0}p^{-1/\alpha}\|X\|_p \le \|X\|_{\psi_\alpha}\le D_\alpha \sup_{p>0}p^{-1/\alpha}\|X\|_p.   
\end{equation}
The (quasi-)norm $\|\cdot\|_{\psi_\alpha}$ is an example of an Orlicz (quasi-)norm. For $\alpha\ge 1$, $\|\cdot\|_{\psi_\alpha}$ is a norm, but $\psi_\alpha(\cdot)$ is not convex for $\alpha<1$. Nevertheless, some properties hold for all $\alpha$:
\begin{enumerate}
    \item $\|X\|_{\psi_\alpha} = 0$ if and only if $\Pr(X=0)=1$;
    \item for any constant $c$, $\|cX\|_{\psi_\alpha} = |c|\|X\|_{\psi_\alpha}$;
    \item $\|X+Y\|_{\psi_\alpha} \le C_\alpha \left(\|X\|_{\psi_\alpha} +  \|Y\|_{\psi_\alpha}\right)$, where $C_\alpha := 2^{1/\alpha}$ if $\alpha<1$ and $C_\alpha:= 1$ if $\alpha\ge 1$ \citep[Lemma A.3]{gss2021}.
\end{enumerate}
A direct consequence of this result is that $\|\cdot\|_{\psi_\alpha}$ is \textit{nearly} convex in the sense that
\[
    \|a X + (1-a) Y\||_{\psi_\alpha} \le C_\alpha \left[a\|X\|_{\psi_\alpha}+(1-a)\|Y\|_{\psi_\alpha}\right].
\]
Another consequence is the concentration that follows after a simple modification of \citet[Lemma 2.2.2][]{vVjW1996book}. Let $X_1,...,X_n$ be sub-Weibull random variables, the maximum of these random variables is also sub-Weibull($\alpha$) with:
\begin{equation}\label{eq:maxbndmod}
\|\max_i |X_i|\|_{\psi_\alpha} \le c_1 \psi_\alpha^{-1}(2n)\max_i \|X_i\|_{\psi_\alpha},   
\end{equation}
where $c_1 = 2/\log(1.5)$. An important tail bound that will be used henceforth follows after Markov's inequality:
\begin{equation}\label{eq:orlicz-tail}
    \Pr\left(\max_{1\le i\le n}|X_i| > x\right) \le \exp\left(-\frac{x^\alpha}{(c_1\max_{1\le i\le n}\|X_i\|_{\psi_\alpha})^\alpha \log(1+2n)}\right).
\end{equation}

This definition extends to vector-valued random variables.
\begin{definition}[Sub-Weibull($\alpha$) random vector]
    Let $\alpha>0$. A random vector $\X\in\R^n$ is a sub-Weibull($\alpha$) random vector if all its scalar projections are sub-Weibull($\alpha$), i.e., $\|\X\|_{\psi_\alpha}:=\sup_{\bb'\bb\le 1}\|\bb'\X\|_{\psi_\alpha}<\infty$.
\end{definition}
It follows from the definition that the individual elements of a sub-Weibull random vector are also sub-Weibull of the same order.

\begin{proof}[Derivation of equation \eqref{eq:maxbndmod}]
    This equation follows after a small modification to the proof of \citet[Lemma 2.2.2][]{vVjW1996book}. We will replicate the proof for the sake of completeness.

    We have to show that for $x,y\ge K>0$, $\psi_\alpha(x)\psi_\alpha(y)\le \psi_\alpha(cxy)$ for some $c>0$:
    \begin{align*}
        (e^{x^\alpha}-1)(e^{y^\alpha}-1) 
        &= e^{x^\alpha+y^\alpha}+1-e^{x^\alpha}-e^{y^\alpha}\\
        &\le e^{(1/x^\alpha+1/y^\alpha)\,(xy)^\alpha}-1\\
        &\le \psi\left(\frac{2}{(\log K)^\alpha}\,xy\right),
    \end{align*}
    for all $\alpha >0$ and $K>0$, and with $c = 2/K^\alpha$. It follows that for all $x,y\ge K$, $\psi_\alpha(x/y)\le\psi_\alpha(cx)/\psi_\alpha(y)$. So, for any $C>0$ and $y\ge K$
    \begin{align*}
        \max_i\psi_\alpha\left(\frac{|X_i|}{Cy}\right) 
        &\le \max_i\left[ \frac{\psi_\alpha(c|X_i|/C)}{\psi_\alpha(y)} + \psi_\alpha\left(\frac{|X_i|}{Cy}\right)I\left(\left|\frac{|X_i|}{C}\right|\le K\right)\right]\\
        &\le\sum_{i=1}^n\frac{\psi_\alpha(c|X_i|/C)}{\psi_\alpha(y)} + \psi_\alpha(K).
    \end{align*}
    Let $C=c\max_i\|X_i\|_{\psi_\alpha}$, $y=\psi_\alpha^{-1}(2n)$, $K=(\log 1.5)^{1/\alpha}$ and take expectation on both sides:
    \[
    \E\left[\max_i\psi_\alpha\left(\frac{|X_i|}{Cy}\right)\right] \le \frac{n}{\psi_\alpha(y)}+\psi_\alpha(K) = 1.
    \]
    Therefore, 
    \[
    \|\max_i|X_i|\|_{\psi_\alpha}\le \frac{2}{\log(1.5)} \psi_\alpha^{-1}(2n)\max_i\|X_i\|_{\psi_\alpha}.
    \]
\end{proof}

\subsection{Mixingales}
We characterise the dependence in the process using the moments and conditional moments of the series. This mode of dependence is weak in the sense that it only requires the conditional moments to converge to their marginals in $L_p$ as we condition further in the past. Some classical measures of dependence, such as strong mixing, imply mixingale dependence.
\begin{definition}[Mixingale]
    Let $\{X_t\}$ be a causal stochastic process and $\{\mathcal{F}_t\}$ be an increasing sequence of $\sigma$-fields in such a way that $X_t$ is $\mathcal{F}_t$ measurable. The process $\{X_t\}$ is an $L_p$-mixingale process with respect to $\{\mathcal{F}_t\}$ if there exists a decreasing sequence $\{\rho_m\}$ and a constant $c_t$ satisfying  
    \[
    \left\|\E[X_t|\mathcal{F}_{t-m}]-E[X_t]\right\|_p \le c_t\rho_m.    
    \]
\end{definition}
Let $\{\F_t\}$ denote an increasing sequence of $\sigma$-algebras and $\{\X_t\}$ denote a stochastic process taking values on $\R^n$ such that $\X_t$ is $\F_t$-measurable for each $t$. The power of mixingales comes from the telescoping sum representation. We can write any random variable as a sum of martingale difference terms and a conditional expectation. Using a telescoping sum representation, we obtain \citep[Section 16.2]{jD1994},
\begin{align}
    \X_t -\E[\X_t] &= \sum_{i=1}^{m}(\E[\X_t|\mathcal{F}_{t-i+1}] - \E[\X_t|\F_{t-i}] ) + (\E[\X_t|\F_{t-m}]) - \E[\X_t]) \nonumber \\
    &= \sum_{i=1}^m V_{i,t} + (\E[\X_t|\F_{t-m}]) - \E[\X_t]),\label{eq:telescoping}
\end{align}
where $V_{i,t} = \E[\X_t|\mathcal{F}_{t-i+1}] - \E[\X_t|\F_{t-i}]$ is a martingale difference process, by construction. Hence, 
\begin{equation}\label{eq:sumXt}
    \sum_{t=1}^T \left(\X_t-\E[\X_t]\right) = \sum_{i=1}^m \left(\sum_{t=1}^T V_{i,t}\right) + \sum_{t=1}^T \E[\X_t - \E(\X_t)|\F_{t-m}].
\end{equation}
In the following lemma we bound the sup norm of the last term in \eqref{eq:sumXt} using the mixingale property.
\begin{lemma} Let $\{\X_t = (X_{1t},...,X_{nt})'\}$ be a causal stochastic process taking values on $\R^n$ and $\{\F_t\}$ an increasing sequence of $\sigma$-fields in such a way that $\X_t$ is $\F_t$ measurable. Suppose that each $\{X_{jt},\F_t\}$ is $L_p$-mixingale with constants $c_{jt}$ and $\{\rho_{jm}\}$ and let $\rho_m = \max_{1\le j\le n}\rho_{jm}$. Then, 
    \[
        \Pr\left(\left|\sum_{t=1}^T \E[\X_t - \E(\X_t)|\F_{t-m}]\right|_\infty>\frac{Tx}{2}\right) \le \frac{2^p}{x^p}n\rho_m^p\bar{c}_T,
    \] 
    where $\bar{c}_T = \max_{1\le j\le n}T^{-1}\sum_{t=1}^Tc_{jt}^p$.
    \label{l:bnddepmixingale}
\end{lemma}
\begin{proof}[Proof of Lemma \ref{l:bnddepmixingale}]
    It follows from the union bound and Markov's inequality that
    \[
        \Pr\left(\left|\sum_{t=1}^T \E[\X_t - \E(\X_t)|\F_{t-m}]\right|_\infty>\frac{Tx}{2}\right) \le n\max_{1\le j\le n}\frac{2^p}{x^p T^p}\E\left|\sum_{t=1}^T \E[X_{jt} - \E(X_{jt})|\F_{t-m}]\right|^p.
    \]
    Applying Loeve's $c_r$ inequality
    \[
        \E\left|\sum_{t=1}^T \E[X_{jt} - \E(X_{j1})|\F_{t-m}]\right|^p \le T^{p-1} \sum_{t=1}^T \E\left|\E[X_{jt} - \E(X_{jt})|\F_{t-m}]\right|^p\le T^p \rho_{jm}^p \bar{c}_T.
    \]
    Result follows directly.
\end{proof}

\subsection{Concentration inequality}
We are interested in bounding the $\ell_\infty$ vector norm of $\sum_{t=1}^T (\X_t-\E[\X_t])$. Then, follows from equation \eqref{eq:sumXt} and simple probability manipulation:
\begin{equation}
    \begin{split}
    \Pr\left(\left|\sum_{t=1}^T\left( \X_t-\E[\X_t]\right)\right|_\infty > Tx\right) 
    &\le \sum_{i=1}^m\Pr\left(\left|\sum_{t=1}^T V_{i,t}\right|_\infty > \frac{Tx}{2m}\right) \\
    & + \Pr\left(\left|\sum_{t=1}^T \E[\X_t - \E(\X)|\F_{t-m}]\right|_\infty>\frac{Tx}{2}\right).
    \end{split}
    \label{eq:probbnd}
\end{equation}
The first term requires bounding a martingale process, whereas we handle the last term by the mixingale concentration in Lemma \eqref{l:bnddepmixingale}.

\citet{mmm2022} derive a concentration bound for the sup norm of vector-valued martingale processes which we adapt to the sub-Weibull case.
\begin{lemma}[Concentration bounds for high dimensional martingales]\label{l:bndmartingale}
    Let $\{\b\xi_t\}_{t=1,...,T}$ denote a multivariate martingale difference process with respect to the filtration $\mathcal{F}_{t}$ taking values on $\R^n$ and assume that $\max_{i,t}\|\xi_{it}\|_{\psi_\alpha}<c_{\psi_{\alpha}}$. Then,
    \[
    \Pr\left(\left|\sum_{t=1}^T\b\xi_t\right|_\infty > T x\right) \le  2n\exp\left(-\frac{Tx^2}{2M^2+xM}\right) + 4\exp\left(-\frac{M^\alpha}{c_1\log(3nT)}\right),
    \]
    for all $M>0$ and $c_1 = (2c_{\psi_{\alpha}}/\log(1.5))^\alpha$ .
\end{lemma}
\begin{proof}[Proof of Lemma \ref{l:bndmartingale}]
    The proof of this lemma follows \cite[Lemma 5]{mmm2022}. Write $\b\xi_t= (\xi_{1t},...,\xi_{nt})'$. The proof follows after applying \citet[Corollary 2.3][]{xFiGqL2012}.

    Write $V_k^2(M) = \max_{1\le i\le n}\sum_{t=1}^k\E[\xi_{it}^2I(\xi_{it}<M)|\mathcal{F}_t]$, $X_{ik} = \sum_{t=1}^k\xi_{it}$ and $X_{ik}'(M) = \sum_{t=1}^k \xi_{it}I(\xi_{it}\le M)$. It follows that for $v>0$ and $x>0$,
    \begin{align*}
        \Pr(|\X_T|_\infty>x) &\le \Pr(\exists i,k:\,X_{ik} > x \cap V_k^2(M)\le v^2) + \Pr(V_T^2(M)>v^2)\\
        &\le \Pr(\exists i,k:\,X_{ik}'(M) > x \cap V_k^2(M)\le v^2) + \Pr(V_T^2(M)>v^2)\\
        &\quad + \Pr\left( \max_{1\le i\le n}\sum_{t=1}^k \xi_{it}I(\xi_{it}> M) >0\right)\\
        &\overset{(a)}{\le} n\exp\left(-\frac{(Tx/M)^2}{2((v/M)^2+\frac{T}{3}x/M)}\right) + \Pr(V_n^2(M)>v^2) \\
        &\quad+ \Pr\left(\max_{1\le t\le T} |\b\xi_t|_\infty > M\right)\\
        &\overset{(b)}{\le} n\exp\left(-\frac{Tx^2}{2M^2+Mx}\right) + 2 \Pr\left(\max_{1\le t\le T} |\b\xi_t|_\infty > M\right) .
    \end{align*}
    In $(a)$ we use the union bound and \citet[Theorem 2.1][]{xFiGqL2012} and in $(b)$ we set $v^2 = T(M^2+\frac{1}{6T}Mx)$ and the following:
    \begin{align*}
        \Pr(V_T^2(M)>v^2)
        &\le \Pr\left(\max_{1\le i\le n}\sum_{t=1}^T\E[\xi_{it}^2I(|\xi_{it}| \le M)|\mathcal{F}_t]\ge v^2\right)\\
        &\quad  + \Pr\left(\max_{1\le i\le n}\sum_{t=1}^T\E[\xi_{it}^2I(\xi_{it}<-M)|\mathcal{F}_t]>0\right)\\
        &\le \Pr\left(\max_{1\le i\le n}\sum_{t=1}^T\E[\xi_{it}^2I(|\xi_{it}| \le M)|\mathcal{F}_t]\ge T(M^2+\frac{1}{6T}Mx)\right)\\
        &\quad + \Pr\left(\max_{1\le t\le T}|\b\xi_t|_\infty>M\right)\\
        &\le \Pr\left(\max_{1\le t\le T}|\b\xi_t|_\infty>M\right),
    \end{align*}
    where in the last line we note $\sum_{t=1}^T\E[\xi_{it}^2I(|\xi_{it}| \le M)|\mathcal{F}_t] \le TM^2$.

    Now, write $\Pr(|\X_T|_\infty\ge Tx) = \Pr(\max_{i\le n} X_{iT}\ge Tx)+\Pr(\max_{i\le n}(-X_{iT})\ge Tx)$ and apply the above development in both terms.

    Finally, use equation \eqref{eq:orlicz-tail} to conclude the proof.
\end{proof}

We are finally ready to introduce the first concentration inequality. It is called a \emph{triplex} inequality \citep{jiang2009} as it has three terms: the first one is a Bernstein-type bound, the second handles the tail, and the third the dependence. 
\begin{theorem}[Concentration for sub-Weibul mixingale processes]\label{t:triplex}
    Let $\{\X_t = (X_{1t},...,X_{nt})'\}$ be a causal stochastic process and $\{\F_t\}$ an increasing sequence of $\sigma$-algebras such that $\X_t$ is $\F_t$ measurable. 
    Suppose that each $\{X_{jt},\F_t\}$ is $L_p$-mixingale with constants $c_{jt}$ and $\{\rho_{jm}\}$, and let $\rho_m = \max_{1\le j\le n}\rho_{jm}$ and $\bar{c}_T = \max_{1\le j\le n}T^{-1}\sum_{t=1}^Tc_{jt}$.  Furthermore, suppose that $\max_{i,t}\|X_{it}\|_{\psi_\alpha} < c_{\psi_\alpha} <\infty$.
    Then, for any natural $m$ and scalar $M>0$: 
    \begin{equation}
        \label{eq:triplex}
        \begin{split}
        \Pr\left(\left|\sum_{t=1}^T \X_t-\E[\X_t]\right|_\infty > Tx\right) &\le 2mn\exp\left(-\frac{Tx^2}{8(Mm)^2+2x Mm}\right)\\
        & + 4m\exp\left(-\frac{M^\alpha}{c_1\log(3nT)}\right)+\frac{2^p}{x^p}n\rho_m^p\bar{c}_T,
        \end{split}
    \end{equation}
    where $c_1 := (2c_{\psi_{\alpha}}/\log(1.5))^{\alpha}$.
\end{theorem}
\begin{proof}
    The proof follows after the direct application of Lemma \ref{l:bnddepmixingale} and Lemma \ref{l:bndmartingale} to Equation \eqref{eq:probbnd}.
\end{proof}

In the next corollary we replace $M$ and $m$ by sequences depending on $T$and $n$ and impose a \emph{sub-Weibull} decay to $\rho_m$.
\begin{corollary}[Sub-Weibull Concentration]\label{c:bndsw}
    According to the assumptions of Theorem \ref{t:triplex}, suppose that $\rho_m \le e^{-m^\gamma/(pc_\rho)}$ for some $\gamma>0$. Then, for any $\tau>0$ and all $T\ge \log(n)+\tau$:
    \begin{equation}
        \begin{split}
        \Pr\left(\left|\sum_{t=1}^T \X_t-\E[\X_t]\right|_\infty > Tx\right) 
        &\le 2c_\rho^{\frac{1}{\gamma}} (\log(n) + \tau)^{\frac{1}{\gamma}}e^{-\tau}+2^p\bar{c}_Tx^{-p}e^{-\tau}\\
        & + 4c_\rho^{\frac{1}{\gamma}}(\log(n) + \tau)^{\frac{1}{\gamma}} e^{-\frac{(x\sqrt{T})^\alpha}{ c_1\log(3nT)(\log(n)+\tau)^{\frac{\alpha}{2}+\frac{\alpha}{\gamma}} } },
        \end{split}
    \end{equation}
    where $c_1:=(8c_{\psi_{\alpha}}c_\rho^{1/\gamma}/\log(1.5))^{\alpha}$.
\end{corollary}
\begin{proof}
    The proof follows by setting $m := c_\rho^{\frac{1}{\gamma}}(\log(n)+\tau)^{\frac{1}{\gamma}}$ and $mM := x\sqrt{T}/4\sqrt{\log(n)+\tau}$.
\end{proof}

\section{Theorems and proofs}
\setcounter{equation}{0}
In this section, we present results of Sections 3, 4 and 5 of the main paper followed by the respective proofs. 

\subsection{Linear processes}
\begin{lemma}
    \label{l:moments}
    Let $\{\b a_j\}$ denote a sequence of elements in $\R^n$ each with finite $L_1$ norm, $\{\Z_t\}$ a sequence of random vectors satisfying $\sup_{|\b a|_1\le 1}\|\b a'\Z_t\|_\psi\le c_t <\infty$ where $\|\cdot\|_\psi$ is some norm and $c_t$ positive constants, and let $W_t = \sum_{j=0}^\infty \b a_j'Z_{t-j}$. Then, $\|W_t\|_\psi \le \sum_{j=0}^\infty |\b a_j|_1c_{t-j}$, provided $\sum_{j=0}^\infty|\b a_j|_1<\infty$.
\end{lemma}
\begin{proof}[Proof of Lemma \ref{l:moments}]
    The proof follows from the convexity of the norms and $\sum_{j=0}^\infty |\b a_j|_1 <\infty$.
\begin{align*}
    \|W_t\|_\psi
        &= \|\sum_{j=0}^\infty \b a_j'\Z_{t-j}\|_\psi \\
        &=\sum_{j=0}^\infty |\b a_j|_1\,\left\|\sum_{j=0}^\infty\frac{|\b a_j|_1}{\sum_{j=0}^\infty |\b a_j|_1}\frac{\b a_j'\Z_{t-j}}{|\b a_j|_1}\right\|_\psi\\
        &\le \sum_{j=0}^\infty |\b a_j|_1\,\sum_{j=0}^\infty\frac{|\b a_j|_1}{\sum_{i=0}^\infty |\b a_i|_1}\left\|\frac{\b a_j}{|\b a_j|_1}'\Z_{t-j}\right\|_\psi\\
        &\le \sum_{j=0}^\infty|\b a_j|_1\sup_{|\b a|_1\le 1}\left\|\b a'\Z_{t-j}\right\|_\psi\\
        &\le \sum_{j=0}^\infty |\b a_j|_1 c_{t-j}
\end{align*}
\end{proof}

\begin{theorem}[Concentration inequality for Linear Processes]\label{t:bndlp}
    Let $\{\X_t = (X_{1t},...,X_{nt})' \}$ be a centred sub-Weibull($\alpha$) causal process taking values in $\R^n$, with sub-Weibull constant $c_{\psi_{\alpha}}$, and let $\{\F_t\}$ be an increasing sequence of $\sigma$-algebras such that $\X_t$ is $\F_t$ measurable. Assume that, for each $i=1,...,n$, $\{X_{it},\F_t\}$ is $L_p$-mixingale with positive constants $\{c_{it}\}$ and decreasing sequence $\{\rho_{im}\}$ and write $\bar{c}_T = \max_{1\le i\le n} T^{-1}\sum_{t=1}^Tc_{jt}$ and $\rho_m = \max_{1\le i\le n}\rho_{im}$.

    Write the linear process $\Y_t = C(L)\X_t$, where $\{C_j\}$ is a sequence of square matrices that satisfy $\max_{1\le i\le n}\sum_{j=1}^\infty j|\b{e}_i'C_j|_1\le \tilde{c}_\infty<\infty$ and denote $c_\infty = \mn{C(1)}_\infty$.

    Then, for any $0<a<1$, $T>0$, $M>0$ and $m=1,2,...$, we have
    \begin{equation}\label{eq:bndlp}
        \begin{split}
            \Pr\left(\left|\sum_{t=1}^T\Y_t\right|_\infty \ge Tx\right)
            &\le 2mn\exp\left(-\frac{T(ax)^2}{8c_\infty^2(Mm)^2+2c_\infty ax Mm}\right)\\
            & + 4m\exp\left(-\frac{M^\alpha}{c_1\log(3nT)}\right)+\frac{(2c_\infty)^p}{(ax)^p}n\rho_m^p\bar{c}_T\\
            & + \exp\left(-\frac{((1-a)Tx)^\alpha}{c_2 \log(1+2n)}\right)
        \end{split}
    \end{equation}
    where $c_1:=(2 c_{\psi_{\alpha}}/\log(1.5))^\alpha$ and $c_2:= (2 c_{\psi_{\alpha}} \tilde{c}_\infty/\log(1.5))^\alpha$
\end{theorem}
\begin{proof}[Proof of Concentration inequality for Linear Processes]
    Using BN decomposition, $\Y_t = C(1)\X_t + \widetilde\X_{t-1}-\widetilde\X_t$, which means that
    \[
        \sum_{t=1}^T\Y_t = C(1)\sum_{t=1}^T\X_t + \widetilde\X_0 - \widetilde\X_T.
    \]
    Hence, 
    \begin{align*}
        \Pr\left(\big|\sum_{t=1}^T\Y_t\big|_\infty > Tx\right) 
        &=  \Pr\left(\big|C(1)\sum_{t=1}^T\X_t + \widetilde\X_0 - \widetilde\X_T\big|_\infty > Tx\right)\\
        &\le  \Pr\left(\big| C(1)\sum_{t=1}^T\X_t|_\infty > \alpha Tx\right) +  \Pr\left(\big|\widetilde\X_0 - \widetilde\X_T\big|_\infty > (1-\alpha) Tx\right)\\
        &:= A+B
    \end{align*}

    We bound $A$ using the triplex inequality for mixingales and we bound $B$ using the ``Orlicz bound'' in equation \eqref{eq:orlicz-tail}.

    It follows directly that $\left|C(1)\sum_t\X_t\right|_\infty\le \mn{C(1)}_\infty|\sum_t\X_t|_\infty\le c_\infty |\sum_t\X_t|_\infty$. Then, we apply the Theorem \ref{t:triplex}:
    \begin{align*} 
        \Pr\left(\big|\ C(1)\sum_{t=1}^T\Y_t|_\infty > a Tx\right) 
        &\le  \Pr\left(\big|\sum_{t=1}^T\X_t|_\infty > a c_\infty^{-1}Tx\right)\\
        &\le 2mn\exp\left(-\frac{T(ax)^2}{8c_\infty^2(Mm)^2+2c_\infty ax Mm}\right)\\
        & + 4m\exp\left(-\frac{M^\alpha}{c_1\log(3nT)}\right)+\frac{(2c_\infty)^p}{(ax)^p}n\rho_m^p\bar{c}_T,
    \end{align*}
     In $B$ we start by writing $\widetilde\X_0-\widetilde\X_T = \widetilde{C}(L)(\X_0-\X_T)$, so that $|\widetilde\X_0-\widetilde\X_T|_\infty = \max_j|\b e_j'\widetilde{C}(L)(\X_0-\X_T)|$. It follows from the triangle inequality and Lemma \ref{l:moments} that $\|\b e_j'\widetilde{C}(L)(\X_0-\X_T)\|_{\psi_\alpha} \le 2c_{\psi_{\alpha}}\max_j\sum_{k=1}^\infty|\b e_j'\widetilde{C}_j|\le 2c_{\psi_{\alpha}} \tilde{c}_\infty$ which is bounded by assumption. Therefore, we have from equation \eqref{eq:orlicz-tail} that
     \[
        \Pr\left(\big|\widetilde\X_0 - \widetilde\X_T\big|_\infty > (1-a) Tx\right) <  \exp\left(-\frac{((1-a)Tx)^\alpha}{c_2 \log(1+2n)}\right),
    \] 
    with $c_2:=(4c_{\psi_{\alpha}}\tilde{c}_\infty/\log(1.5))^\alpha$
\end{proof}

\subsection{Empirical lag-h autocovariance matrices}
Let $\{\X_t\}$ denote a centred, sub-Weibull, causal stochastic process taking values on $\R^n$, and $\Y_t = C(L)\X_t$, $t=1,\cdots,T$, a dependent sequence of random vectors.  Let
\begin{equation}\label{eq:cov}
    \widehat{\Gamma}_T(h) := \frac{1}{T}\sum_{t=h+1}^T\Y_t\Y_{t-h}'\quad\mbox{and}\quad\Gamma_T(h) := \E[\widehat{\Gamma}_T(h)] = \frac{1}{T}\sum_{t=h+1}^T\E[\Y_t\Y_{t-h}'].
\end{equation}  
Our goal is to find a bound for 
\begin{equation}\label{eq:delta}
    \Delta_T(h) = \mn{\widehat{\Gamma}_T(h) - {\Gamma}_T(h)}_{\max} = \left|\vec(\widehat{\Gamma}_T(h) - {\Gamma}_T(h))\right|_\infty,
\end{equation}
which is the maximum element in absolute value of the matrix.

\begin{theorem}\label{t:acovmatrix}
    Let $\{\X_t = (X_{1t},...,X_{nt})' \}$ be a centred sub-Weibull($\alpha$), causal process taking values in $\R^n$, and let $\{\F_t\}$ be an increasing sequence of $\sigma$-algebras such that $\X_t$ is $\F_t$ measurable. 

    Write $\bs\eta_t(k) = \vec(\X_t\X_{t-k}')$ and the stochastic process $\{\bs\eta_t(k) = (\eta_{1t}(k),...,\eta_{n^2t}(k))'\}$ for $k=0,1,...$. The processes $\{\eta_{it}(k),\F_t\}$ are $L_1$-mixingale with constants $c_{it}$ and decreasing sequences $\rho_{im}$, for each $k=0,1,...$ and $i=1,...,n^2$. Let $\bar{c}_T = \max_{1\le i\le n^2} T^{-1}\sum_{t=1}^Tc_{it}$ and $\rho_m = \max_{1\le i\le n}\rho_{im}\le e^{-m^\gamma/c_\rho}$.
    
    Finally, let the linear process $\Y_t = C(L)\X_t$, where $\{C_j\}$ is a sequence of square matrices and define the following finite constants:
    \begin{enumerate}
        \item[a. ] $\tilde{c}_{2,\infty}:= \sum_{j=1}^\infty j \mn{C_j}^2$;
        \item[b. ] $c_h := \max_{1\le k\le h}\mn{\sum_{j=0}^\infty C_{j+k}\otimes C_j}_\infty$;
        \item[c. ] $c_\infty:= \max_{1\le k\le h}\mn{\sum_{i=0}^\infty(\sum_{j=i+k}^\infty C_j)\otimes C_i}_\infty$;
        \item[d. ] $\tilde{c}_\infty := \max_{1\le i\le n^2}\sum_{j=1}^\infty j\left|\sum_{k=0}^\infty\bs{e}_i'(C_{j+k}\otimes C_k)\right|_1$.
    \end{enumerate}
    Let $\Delta_T(h)$ be as defined in equations \eqref{eq:cov} - \eqref{eq:delta}. Then, for each $n$ and $T$ that satisfies $T>4\log(n)$ and $3T <n^{\beta-1}$ for some $\beta>1$: 
    \begin{equation}
        \begin{split}
        \Pr(\Delta_T(h) > x) 
        &\le c_1\frac{\log(n)^{1/\gamma}}{n^2} + \frac{c_2\bar{c}_T}{n^2x} + c_3\log(n)^{1/\gamma}e^{-\frac{(x\sqrt{T})^{\alpha/2}}{c_4\log(n)^{1+\frac{\alpha}{4}+ \frac{\alpha}{2\gamma}}}} + 4e^{-\frac{(xT)^{\alpha/2}}{c_5\log(n)}} 
        \end{split},
    \end{equation}
    where the constants $c_1,...,c_5$ depend only on $\beta$, $c_\rho, c_{\psi_{\alpha/2}}, \tilde{c}_{2,\infty}, c_h, c_\infty, \tilde{c}_\infty$, but not on $n$ or $T$. 
\end{theorem}
\begin{proof}[Proof of Theorem \ref{t:acovmatrix}]

    The strategy is to rewrite the autocovariance matrix as sums of linear and sub-Weibull processes and use the above inequalities. Let $C_j=0$ for all $j<0$ and write the outer product of linear processes
\begin{align*}
    \Y_t\Y_{t+h}' 
    & = C(L)\X_t\X_{t+h}'C(L)'\\
    & = \left(\sum_{j=0}^\infty C_{j-h}\X_{t+h-j}\right)\left(\sum_{j=0}^\infty C_j\X_{t+h-j}\right)'\\
    & = \sum_{j=0}^\infty C_j\X_{t-j}\X_{t-j}C_{j+h}'\\
    &\quad + \sum_{k=1}^\infty\sum_{j=0}^\infty C_j\X_{t-j}\X_{t-k-j}'C_{j+h+k}'\\
    &\quad + \sum_{k=1}^\infty\sum_{j=0}^\infty C_{j-h+k}\X_{t+h-k-j}\X_{t+h-j}'C_j'.
\end{align*}
Recall that for any matrices $A$, $B$ and $C$, $\vec(ABC') = (C\otimes A)~\vec(B)$, where $C\otimes A$ is the Kronecker product of matrices $A$ and $C$, and there is a commutation matrix $\bs{P}$ such that $\vec(A') = \bs{P}\vec(A)$.

Let $\bs{\eta}_t(k) = \vec(\X_t\X_{t-k}')$ and $F_{j,k} = C_j\otimes C_k$, then
\begin{align*}
    \Z_{t,k}(h) 
    &:= \vec\left(\sum_{j=0}^\infty C_j\X_{t-j}\X_{t-k-j}'C_{j+h+k}'\right)\\
    & = \sum_{j=0}^\infty (C_{j+h+k}\otimes C_j)~\vec(\X_{t-j}\X_{t-k-j})\\
    & = \sum_{j=0}^\infty F_{j+h+k,j}~L^j\bs{\eta}_t(k)\\
    & = F_{k+h}(L)\bs{\eta}_t(k).
\end{align*}
Now, for some commutation matrix $\bs{P}$
\begin{align*}
    \Z^*_{t+h,k}(-h) 
    &:= \vec\left(\sum_{j=0}^\infty C_{j-h+k}\X_{t+h-k-j}\X_{t+h-j}'C_j'\right)\\
    & = \bs{P} \vec\left[\left(\sum_{j=0}^\infty C_{j-h+k}\X_{t+h-k-j}\X_{t+h-j}'C_j'\right)'\right] \\
    & = \bs{P} \vec\left(\sum_{j=0}^\infty C_j\X_{t+h-j}\X_{t+h-k-j}'C_{j-h+k}'\right) \\
    & = \sum_{j=0}^\infty \bs{P} F_{j-h+k,j}~L^j\bs{\eta}_{t+h}(k)\\
    & = \bs{P} F_{k-h}(L)\bs{\eta}_{t+h}(k)\\
    & = \bs{P} \Z_{t+h,k}(-h),
\end{align*}
so that
\begin{equation} \label{eq:vec}
    \vec(\Y_t\Y_{t-h}') = \Z_{t,0}(h) + \sum_{k=1}^\infty \left[ \Z_{t,k}(h) + \Z^*_{t-h,k}(-h)\right].
\end{equation}

Using the BN decomposition on $\Z_{t,k}(h)$ and $\Z^*_{t,k}(h) = \bs{P}\Z_{t,k}(h)$:
\[
  \Z_{t,k}(h) = F_{k+h}(L)\bs{\eta}_t(k) = F_{k+h}(1)\bs{\eta}_t(k) - (1-L) \tilde{F}_{k+h}(L)\bs{\eta}_t(k),
\]
where
\[
    \tilde{F}_{k+h,j} = \sum_{s=j+1}^\infty F_{k+h+s,s}\quad\mbox{and}\quad\tilde{F}_{k+h}(L) = \sum_{j=0}^\infty\tilde{F}_{k+h,j}L^j.
\] 
Similarly,
\[
    \Z^*_{t-h,k}(h) = \bs{P}\,F_{k-h}(1)\bs{\eta}_{t+h}(k) - (1-L) \bs{P}\,\tilde{F}_{k-h}(L)\bs{\eta}_{t+h}(k) .
\]
Bringing everything together:
\begin{align*}
    \vec(\Y_t\Y_{t+h}') 
    & = F_h(1)\bs{\eta}_t(0) + \sum_{k=1}^\infty F_{k+h}(1)\bs\eta_t(k) + F_{k-h}(1)\bs\eta_{t+h}(k)\\
    &\quad -(1-L)\left\{\tilde{F}_h(L)\bs\eta_t(0)+\sum_{k=1}^\infty\tilde{F}_{k+h}(L)\bs\eta_t(k)\right\}\\
    &\quad -(1-L)\sum_{k=1}^\infty\bs{P}\,\tilde{F}_{k-h}(L)\bs\eta_{t+h}(k).
\end{align*}
Set $\bar{\bs\eta}_{t}(h) = \bs\eta_{t}(h) - \E\bs\eta_{t}(h)$. The centered empirical autocovariances can be written as
\begin{align*}
    \vec(\widehat{\Gamma}_T(h) - \Gamma_T(h)) & =  \frac{1}{T}\sum_{t=1}^{T-h}\vec(\Y_t\Y_{t+h}' - \E\Y_t\Y_{t+h}')\\
    & = \frac{1}{T}\sum_{t=1}^{T-h} \sum_{k=0}^\infty F_{k+h}(1)\tilde{\bs\eta}_t(k)\\
    & \quad + \frac{1}{T}\bs{P}\sum_{t=1}^{T-h} \sum_{k=1}^\infty F_{k-h}(1)\tilde{\bs\eta}_{t+h}(k)\\
    & \quad + \frac{1}{T}\sum_{k=0}^\infty \tilde{F}_{k+h}(L)\tilde{\bs\eta}_0(k) - \frac{1}{T}\sum_{k=0}^\infty \tilde{F}_{k+h}(L)\tilde{\bs\eta}_{T-h}(k) \\
    & \quad + \frac{1}{T}\bs{P}\,\sum_{k=1}^\infty\tilde{F}_{k-h}(L)\tilde{\bs\eta}_h(k) - \frac{1}{T}\bs{P}\,\sum_{k=0}^\infty \tilde{F}_{k-h}(L)\tilde{\bs\eta}_{T}(k).
\end{align*}
The BN decomposition transforms the average of a bilinear process $T^{-1}\sum_t\Y_t\Y_{t+h}'$ into the sum of the averages of the linear process and sub-Weibull random vectors. This means we can derive, the concentration bounds for the empirical covariance matrix from concentration bounds for linear process and sub-Weibull tail bounds. We write
\begin{align}
    \Pr(\Delta_T(h) > x ) 
    & \le \Pr\left(\left|F_{h}(1)\sum_{t=1}^{T-h}\tilde{\bs\eta}_t(0)\right|_\infty > \frac{Tx}{4}\right)\label{eq:pl0}\\
    &\quad + \Pr\left(\left|\sum_{t=1}^{T-h} \sum_{k=1}^\infty F_{k+h}(1)\tilde{\bs\eta}_t(k)\right|_\infty > \frac{Tx}{4}\right)\label{eq:pl1}\\
    &\quad + \Pr\left(\left|\sum_{t=1}^{T-h} \sum_{k=1}^\infty F_{k-h}(1)\tilde{\bs\eta}_{t+h}(k)\right|_\infty > \frac{Tx}{4}\right)\label{eq:pl2}\\
    &\quad + \Pr\left(\left|\sum_{k=0}^\infty \tilde{F}_{k+h}(L)\tilde{\bs\eta}_0(k)-\sum_{k=0}^\infty \tilde{F}_{k+h}(L)\tilde{\bs\eta}_{T-h}(k)\right|_\infty > \frac{Tx}{8}\right)\label{eq:sw1}\\
    &\quad + \Pr\left(\left|\sum_{k=1}^\infty\tilde{F}_{k-h}(L)\tilde{\bs\eta}_h(k)-\sum_{k=1}^\infty \tilde{F}_{k-h}(L)\tilde{\bs\eta}_{T}(k)\right|_\infty > \frac{Tx}{8}\right).\label{eq:sw2}
\end{align}
Let 
\[
    \bs{W}_t = \sum_{k=1}^\infty\tilde{F}_{k-h}(L)\bs{V}_t(k),
\]
where $\{\bs{V}_t(k)\}$ is sub-Weibull($\alpha$) and $\|\delta'\bs{V}_t(k)\|_{\psi_\alpha} <c_{\psi_{\alpha}}$ for every $\delta'\delta = 1$. 
\begin{align*}
    \|\delta'\bs{W}_t\|_{\psi_\alpha}
    &= \left\|\sum_{k=1}^\infty\sum_{j=0}^\infty\delta'\tilde{F}_{k-h,j}\bs{V}_{t-j}(k)\right\|_{\psi_\alpha}\\
    &\le C_\alpha \sum_{k=1}^\infty\sum_{j=0}^\infty\left\|\delta'\tilde{F}_{k-h,j}\bs{V}_{t-j}(k)\right\|_{\psi_\alpha}\\
    &\le C_\alpha  \sum_{k=1}^\infty\sum_{j=0}^\infty\sum_{s=j+1}^\infty\left\|\delta'F_{k-h+s,s}\bs{V}_{t-j}(k)\right\|_{\psi_\alpha}\\
    &\le C_\alpha  \sum_{k=1}^\infty\sum_{j=0}^\infty\sum_{s=j+1}^\infty\mn{C_{k-h+s}\otimes C_s}\sup_{\delta'\delta=1}\left\|\delta'\bs{V}_{t-j}(k)\right\|_{\psi_\alpha}\\
    &\le C_\alpha c_{\psi_{\alpha}} \,\sum_{k=1}^\infty\sum_{j=0}^\infty\sum_{s=j+1}^\infty\mn{C_{k-h+s}}\mn{C_s}\\
    &\le C_\alpha c_{\psi_{\alpha}} \,\sum_{k=1}^\infty\left(\sum_{s=(k-h+1)_+ }^\infty\mn{C_s}^2\right)^{1/2} \cdot \sum_{j=0}^\infty\left(\sum_{s=j+1}^\infty\mn{C_s}^2\right)^{1/2}\\ 
    &\le C_\alpha c_{\psi_{\alpha}} \, \left(\sum_{k=(-h)_+}^\infty (k-h)_+\mn{C_k}^2\right)^{1/2} \cdot \left(\sum_{j=0}^\infty j\mn{C_j}^2\right)^{1/2}\\
    &\le C_\alpha c_{\psi_{\alpha}} \sum_{j=0}^\infty j\mn{C_j}^2.
\end{align*}
Therefore, if $\tilde{c}_{2,\infty}:=\sum_{j=1}^\infty j \mn{C_j}^2<\infty$, $\bs{W}_t$ is sub-Weibull($\alpha$) random variable. It follows from definition that if $\delta'V_t(k)$ is sub-Weibull($\alpha$), then $(\delta'V_t(k))^2$ is sub-Weibull($\alpha/2$). 

We first apply the sub-Weibull tail bound in equation \eqref{eq:orlicz-tail}, for bounding \eqref{eq:sw1} and \eqref{eq:sw2}. Let $\bs{V}_t(k) = \tilde{\bs\eta}_s(k) - \tilde{\bs\eta}_t(k)$:
\[
\|\delta'(\tilde{\bs\eta}_s - \tilde{\bs\eta}_t)|_{\psi_\alpha}\le 2 \|\delta'\tilde{\bs\eta}_s\|_{\psi_\alpha}\le 2 \max_t\|(\bs{b}'\bs\epsilon_t)^2\|_{\psi_\alpha},
\] 
for some $\bs{b}'\bs{b}=1$. Hence:
\begin{equation}\label{eq:bnd-sw1}
    \mbox{\eqref{eq:sw1}} \le \exp\left(-\frac{(Tx)^{\alpha/2}}{c_2 \log(\sqrt{3}n)}\right),\\
\end{equation}
and, similarly, 
\begin{equation}\label{eq:bnd-sw2}
    \mbox{\eqref{eq:sw2}} \le \exp\left(-\frac{(Tx)^{\alpha/2}}{c_2 \log(\sqrt{3}n)}\right),
\end{equation}
where $c_2 := 2(32/\log(1.5) \cdot C_\alpha c_{\psi_{\alpha/2}} \tilde{c}_{2,\infty})^{\alpha/2}$. 

We have to specify the form of dependence of the series. Let $\{\bs\eta_t(k) = (\eta_{1t}(k),...,\eta_{n^2t}(k))'\}$, where $\bs\eta_t(k) = \vec(\X_t\X_{t-k}')$, for $k=0,1,...$, and $\{\F_t\}$ be an increasing sequence of $\sigma$-algebras such that $\X_t\X_{t-k}'$ is $\F_t$ measurable. The processes $\{\eta_{jt}(k),\F_t\}$ are $L_1$ mixingale with constants $c_{jt}$ and decreasing sequences $\rho_{jm}$, for each $k=1,2,...$ and $j=1,...,n^2$. The following bound holds for $k=0,1,2,...$:
\[
    \|\E[\eta_{jt}(k)-\E[\eta_{jt}(k)]|\F_{t-m}]\|_1\le c_{jt}\rho_{jm}.
\]

Let $k=0$, $\rho_m := \max_j\rho_{jm} \le e^{-m^\gamma/c_\rho}$, and recall that for a square matrix $A$ and a vector $b$, $|Ab|_\infty\le\mn{A}|b|_\infty$. We apply the probability bound of Corollary \ref{c:bndsw} (sub-Weibull concentration) on Equation \eqref{eq:pl0}, with $\alpha$ replaced by $\alpha/2$ and $n$ by $n^2$:
\begin{equation}\label{eq:bnd-pl0}
    \mbox{\eqref{eq:pl0}} \le c_1\frac{\log(n)^{1/\gamma}}{n^2} + \frac{8c_h\bar{c}_T}{n^2x}+2c_1\log(n)^{1/\gamma}e^{-\frac{(x\sqrt{T})^{\alpha/2}}{c_2\log(n)^{1+\alpha/4+\alpha/2\gamma}}},
\end{equation}
where $c_1 := 2(4c_\rho)^{1/\gamma}$, $c_h := \mn{F_h(0)}_\infty = \mn{\sum_{j=0}^\infty C_{j+h}\otimes C_j}_\infty$, and for some $\beta>1$ define below, $c_2 := \beta 2^{1+3\alpha+\alpha/\gamma}(c_{\psi_{\alpha/2}}c_\rho^{1/\gamma}/\log(1.5))^{\alpha/2}$.

Now, let $k= 1,2,...$ and $\rho_m := \max_j\rho_{jm} \le e^{-m^\gamma/c_\rho}$. Using the concentration inequality in Equation (2.5)
on Equation \eqref{eq:pl1}, with $\alpha$ replaced by $\alpha/2$, $n$ by $n^2$, and the constants $c_\infty:= \max_{1\le k\le h}\mn{\sum_{i=0}^\infty(\sum_{j=i+k}^\infty C_j)\otimes C_i}_\infty$ and $\tilde{c}_\infty := \max_{1\le i\le n^2} \sum_{j=1}^\infty j \left|\sum_{k=0}^\infty\bs{e}_i'(C_{j+k}\otimes C_k)\right|_1$:
\begin{equation}\label{eq:bnd-pl1}
    \mbox{\eqref{eq:pl1}} \le c_1\frac{\log(n)^{1/\gamma}}{n^2} + \frac{16c_\infty\bar{c}_T}{n^2x}+2c_1
    \log(n)^{1/\gamma}e^{-\frac{(x\sqrt{T})^{\alpha/2}}{c_2\log(n)^{1+\frac{\alpha}{4}+\frac{\alpha}{2\gamma}}}}+e^{-\frac{(xT)^{\alpha/2}}{c_3\log(n)}}.
\end{equation}
where $c_1 := 2^{1+2/\gamma}c_\rho^{1/\gamma}$, $c_2 :=  \beta 2^{1+3\alpha+\alpha/\gamma}(c_{\psi_{\alpha/2}}c_\infty c_\rho^{1/\gamma}/\log(1.5))^{\alpha/2}$, and $c_3:= 3\cdot 2^{2\alpha}(c_{\psi_{\alpha/2}}\tilde{c}_\infty)^{\alpha/2}$. A similar bound holds for \eqref{eq:pl2}.

The probability bounds in Equations \eqref{eq:bnd-pl0} and \eqref{eq:bnd-pl1} hold for all $n$ and $4\log(n)\le T<n^{\beta-1}/3$ for some $\beta>1$. 
\end{proof}

\begin{remark}[Sample size restrictions]
    Above restrictions on $T$ can be removed if one uses Theorem \ref{t:triplex} (concentration for sub-Weibul mixingale processes) and Theorem \ref{t:bndlp} (concentration inequality for linear processes) to obtain the probability bound. 
\end{remark}

\begin{remark}[Mixingale condition]
Under regularity conditions, the process $\bs{\eta}_t(k)$ inherits the dependence properties of $\{X_t\}$. The next lemma is used to obtain sufficient regularity conditions.
\begin{lemma}\label{l:mixingale_prod}
    Let $X$ and $Y$ be uncorrelated, centred random variables and suppose that $\mathcal{G}\subset\F$ are two $\sigma$-algebras of events. Let $\|\E[X|\F]\|_2=\rho_\F$ and $\|\E[Y|\mathcal{G}]\|=\rho_\mathcal{G}$, with the natural extension that $\E[Y|\emptyset] = \E[Y] = 0$. Then, for any $\F$-measurable $M>0$,
    \[
    \|\E[XY|\F]\|_1 \le \|M\|_2\rho_\F + \rho_F\rho_\mathcal{G}+\||XY|I(|Y-\E[Y|\mathcal{G}]|>M)\|_1.
    \]
    Also, if $M$ is deterministic, 
    \[
    \||XY|I(|Y-\E[Y|\mathcal{G}]|>M)\|_1 \le \|XY\|_2\|Y\|_2M^{-1}
    \]
\end{lemma}
\begin{proof}
    Let $W = YI(|Y-\E[Y|\mathcal{G}]|\le M)$ and $Z = Y-W = YI(|Y-\E[Y|\mathcal{G}]|>M)$. Then, $\E[Y|\mathcal{G}]-M<W<\E[Y|\mathcal{G}]+M$. Multiplying by $X$ and taking $\E[\cdot|\F]$, 
    \[
        \E[X|\F]\E[Y|\mathcal{G}]-\E[X|\F]M<\E[XW|\F]<\E[X|\F]\E[Y|\mathcal{G}]+\E[X|\F]M,
    \]
    where we use the fact that $\E[X\E[Y|\mathcal{G}]|\F] = \E[X|\F]\E[Y|\mathcal{G}]$.
    Therefore, $\E|\E[XW|\F] - \E[X|\F]\E[Y|\mathcal{G}]| < \E|M\E[X|\F]|$. Now:
    \begin{align*}
    \|\E[XY|\F]\|_1
        &\le \|\E[XW|\F]\|_1 + \|\E[XZ|\F]\|_1\\ 
        &\le \|\E[XW|\F] - \E[X|\F]\E[Y|\mathcal{G}]\|_1 + \|\E[X|\F]\E[Y|\mathcal{G}]\|_1 + \|XZ\|_1.
    \end{align*}
    It follows from the Cauchy-Schwarz and Markov's inequalities:
    \begin{align*}
         \|XZ\|_1 
            & =\||XY|I(|Y-\E[Y|\mathcal{G}]|>M)\|_1\\
            &\le\|XY\|_2\Pr(|Y-\E[Y|\mathcal{G}]|>M)^{1/2}\\
            &\le \|XY\|_2 M^{-1}\|Y-\E[Y|\mathcal{G}]\|_2\\
            &\le \|XY\|_2\|Y\|_2M^{-1}.
    \end{align*}
\end{proof}

Now, suppose that $\{\X_t\}$ is centred, uncorrelated in time, and $L_2$-mixingale. We must consider two situations: $k\le m$ and $k>m$. If $k\le m$, for any two elements $i$ and $j$, $\|\E[X_{it}|\F_{t-m}]X_{j,t-k}\|_1\le \| \E[X_{it}|\F_{t-m}]\|_2\|X_{j,t-k}\|_2$, which means that $\|\E[X_{it}|\F_{t-m}]X_{j,t-k}\|_1\le \|X_{j,t-k}\|_2c_{it}\rho_m$. Assume now that $k>m$. Let $\F = \F_{t}$, $\mathcal{G} = \F_{t-k-m}$, $X = x_{it}$, and $Y=X_{j,t-k}$ in Lemma \ref{l:mixingale_prod}. Then, $\rho_\F = c_{it}\rho_m$ and $\rho_\mathcal{G} = c_{jt}\rho_m$, where $\rho_m \le e^{-m^\gamma/c_\rho}$. It also follows that
\[
    \||X_{i,t}X_{j,t-k}|I(|X_{j,t-k}-\E[X_{j,t-k}|\F_{t-m-k}]|>M)\|_1 \le \|X_{i,t}X_{j,t-k}\|_2\|X_{j,t-k}\|_2 M^{-1}.
\]
Select $M = \rho_m^{-1/2}$ to obtain the bound
\[
    \|\E[X_{it}X_{j,t-k}|\F_{t-m}]\|_1 \le (c_{it}+c_{it}c_{j,t-k}\rho_m^{3/2}+\|X_{i,t}X_{j,t-k}\|_2\|X_{j,t-k}\|_2) e^{-m^\gamma/2c_\rho}.   
\] 

Therefore, we can select constants $c_t$ in a way that for any $k>0$ and $m>0$, $ \|\E[X_{it}X_{j,t-k}|\F_{t-m}]\|_1 \le c_te^{-m^\gamma/2c_\rho}$, which means that $\{\eta_{it}(k),\F_t\}$ is $L_1$-mixingale with the same ``sub-Weibull'' rate. For $k=0$ we directly assume that $\{X_{it}X_{jt},\F_t\}$ is $L_1$ mixingale. \citet{mmm2022} discusses this condition providing illustrative examples.
\end{remark}

\subsection{LASSO estimation of VAR(p)}
As in \citet{sBgM2015, aKlC2015,wlt2020} and \citet{mmm2022}, the oracle estimation bounds for VAR($p$) LASSO estimation derive from the \emph{deviation bound} and \emph{restricted eigenvalue} conditions. The former imposes a lower bound on the regularisation parameter for a sparse solution, whereas the second condition ensures that the loss is strongly convex in the direction of interest. These conditions are only satisfied in a high probability set and are equivalent to controlling the variation of the maximum entrywise norm of autocovariances $\Gamma_T(h)$ around their means. 

With some abuse of notation, denote $\widehat{\Gamma}_T(h-s) = (T-p)^{-1}\,\sum_{t=p+1}^T\Y_{t-s}\Y_{t-h}'$ so that, $\E[\widehat{\Gamma}_T(h-s)]=\Gamma_T(h-s)$. The \emph{deviation bound} condition requires that
\begin{equation}\label{eq:db}
    \lambda\ge2 \max_{1\le h\le p}\max_{1\le i,j \le n} \left|\frac{1}{T-p}\sum_{t=p+1}^TW_{i,t}Y_{j,t-h}\right| = \frac{2}{T-p}\max_{1\le h\le p}\mn{\sum_{t=p+1}^T\W_t\Y_{t-h}'}_{\max}.
\end{equation}
The population residual $\W_t = \Y_t - \sum_{i=1}^pA_i^*\Y_{t-i}$ is orthogonal to $\Y_{t-h}$. We write the empirical process
\begin{align*}
\frac{1}{T-p}\sum_{t=p+1}^T\W_t\Y_{t-h}' &= (\widehat{\Gamma}_T(h)-\Gamma_T(h)) - A_1^*(\widehat{\Gamma}_T(h-1)-\Gamma_T(h-1)) \\
&\quad + \cdots + A_p^*(\widehat{\Gamma}_T(h-p)-\Gamma_T(h-p)).
\end{align*}
If follows after a simple application of the Hölder's inequality that \eqref{eq:db} is satisfied if
\begin{equation*}
\lambda \ge 2(1+\mn{\bs{A}^*}_\infty)\max_{0\le h\le p}\mn{\widehat{\Gamma}_T(h)-\Gamma_T(h)}_{\max},    
\end{equation*}
where $\bs{A}^* = [A_1,\ldots,A_p]$ is the $n\times np$ matrix of population coefficients.

Let $\widehat{\Sigma}_T$ denote the empirical version of $\Sigma_T$, i.e. the square matrix with blocks $(i,j)$ given by $\widehat{\Gamma}_T(i-j)$. Let $S_\eta:=\{(i,j)\in \{1,...,n\}^2:[\bs{A}^*]_{i,j}\ge\eta\}$ be the set of indices $(i,j)$ in which elements of $\bs{A}^*$ are above some threshold and $C_{S_\eta}$ be the matrix with elements $[C]_{i,j}$ if $(i,j)\in S_\eta$ and zero otherwise. The \emph{restricted strong convexity} condition requires that there are parameters $\kappa$ and $\tau(\bs{A}^*)$ so that for all $\bs{\Delta}\in\{C\in\R^{n\times np}:|\vec(C_{S_\eta})|_1\le|\vec(C_{S_\eta})|_1+4|\vec(C_{S_\eta^c})|_1\}$,
\begin{equation}\label{eq:rsc}
    \tr(\Delta'\widehat{\Sigma}_T\Delta)\ge \kappa \mn{\Delta}^2 - \tau(\bs{A}^*).
\end{equation}

As $\tr(\Delta'\widehat{\Sigma}_T\Delta) = \sum_{i=1}^n\Delta_i'\widehat{\Sigma}_T\Delta_i$ we follow the derivation Lemma 3 in \citet{mmm2022} to verify that \eqref{eq:rsc} holds with $\kappa = \sigma_\Sigma^2/2$ and $\tau(\bs{A}^*) = \sigma_\Sigma^{2(1-q)}R_q\eta^{2-q}/2$ whenever
\begin{equation*}
    \max_{0\le h\le p}\mn{\widehat{\Gamma}_T(h)-\Gamma_T(h)}_{\max}\le\frac{\sigma_\Sigma^2\eta^q}{64R_q}.
\end{equation*}

Finally, the oracle estimation bound below follows as in \citet{mmm2022} Theorem 1. Let $\eta = \lambda/\sigma_\Sigma^2$ and $a_\lambda = \min\left(\frac{\lambda}{2(1+\mn{\bs{A}}_{\infty})},\frac{\sigma_\Sigma^{2(1-q)}\lambda^q}{64R_q}\right)$ and set
\begin{equation}\label{eq:problasso}
        \begin{split}
        \pi(a_\lambda) 
        & = c_1\frac{\log(n)^{1/\gamma}}{n^2} + \frac{c_2\bar{c}_T}{n^2a_\lambda} + c_3\log(n)^{1/\gamma}e^{-\frac{(a_\lambda\sqrt{T})^{\alpha/2}}{c_4\log(n)^{1+\frac{\alpha}{4}+ \frac{\alpha}{2\gamma}}}} + 4e^{-\frac{(a_\lambda T)^{\alpha/2}}{c_5\log(n)}} .
        \end{split}
\end{equation}
Then, with probability $1-2p\pi(a_\lambda)$ and $T>4\log(n)$,
\[
    \sum_{i=1}^p\mn{\widehat{A}_i - A^*_i}_F^2 \le (44+2\lambda)R_q\left(\frac{\lambda}{\sigma_\Sigma^2}\right)^{2-q}.
\]
Suppose Assumption DGP and Assumption Identification hold with $\sigma_\Sigma^2$ and $R_q$ uniformly bounded for all $n$. Eventually, for $\lambda$ sufficiently small, $a_\lambda = \lambda/2(1+\mn{\bs{A}^*}_\infty)$. Set 
\begin{equation}\label{eq:lambda}
    \lambda \ge c\,\left(\log\log(n)+\epsilon\right)^{2/\alpha}\log(n)^{2/\alpha+1/\gamma}\sqrt{\frac{\log({n})}{T}},
\end{equation}
for any $\epsilon>0$ and some constant $c>0$ sufficiently large. Then, the oracle bound will hold with probability at least
\[
1-2p\pi(a_\lambda) = 1-4pe^{-\epsilon}-\frac{2pc_1\log(n)^{1/\gamma}}{n^2} - \frac{(c_2/c)\bar{c}_T\sqrt{T}}{n^2\log(n)^{1/2+2/\alpha+1/\gamma}(\log\log(n)+\epsilon)^{2/\alpha}} .
\]

In practise, $\sigma_\Sigma^2$ and $R_q$ can grow as a function of $n$, in which case the rate of decrease on $\lambda$ would have to accommodate these quantities. We have that for $a_\lambda = {\sigma_\Sigma^{2(1-q)}\lambda^q}/{64R_q}$
\[
    \lambda^q \ge c\,\frac{\sigma_\Sigma^{2(1-q)}}{R_q}\left(\log\log(n)+\epsilon\right)^{2/\alpha}\log(n)^{2/\alpha+1/\gamma}\sqrt{\frac{\log({n})}{T}},
\]
for some constant $c$ sufficiently large. However, it is not necessarily a constraint in the rate of $\lambda$, provided that $\lambda R_q^{1/(1-q)}/\sigma_\Sigma^2 = O(1)$.

\subsection{HAC estimator bound}

\begin{proof}[Derivation of Equation (5.33) in Section 4.2]

We split the error in two parts:
\begin{equation}\label{eq:bias-estimation}
\begin{split}
\mn{\widehat{\bs{F}}(w) -\bs{F}(w)}_{\max} \le &\mn{\widehat{\bs{F}}(w)-\E[\widehat{\bs{F}}(w)]}_{\max}\\&+\mn{\E[\widehat{\bs{F}}(w)]-\bs{F}(w)}_{\max}.
\end{split}
\end{equation}
The first term is the estimation error, and the second term is the bias. The estimation error is 
\begin{equation}\label{eq:estimation}
    \mn{\widehat{\bs{F}}(w)-\E\widehat{\bs{F}}(w)}_{\max} = \frac{1}{2\pi}\mn{\sum_{h=-T+1}^{T-1}\kappa_{g,\epsilon}\left(\frac{h}{M_T}\right)\left[\widehat{\Gamma}_T(h)-\Gamma_T(h)\right]}_{\max}
\end{equation}
and will be bound using Theorem \ref{t:acovmatrix}. The bias term is controlled as in Theorem 2.1 in \citet{p2011}, under distinct assumptions on $\{\mn{\Gamma_j}_{\max}\}$. \citet{p2011} admits a distinct \emph{bandwidth} $M_T$ for each individual element $F_{ij}(w)$ in $\bs{F}$. We consider a simpler and conservative approach, taking $M_T$ the same to all elements. 

Following the same arguments used in Theorem 2.1 of \citet{p2011} and under Assumption DGP, the bias term is bounded by
\begin{equation}\label{eq:F_bias_bound}
    \mn{\E[\widehat{\bs{F}}(w)]-\bs{F}(w)}_{\max} \le \frac{2c_r}{\pi\epsilon^r M_T^r} + \frac{c_0}{\pi T},
\end{equation}
where $c_r = \sum_{|h|>0}|h|^{1+r}\mn{\Gamma_j}_{\max}<\infty$. Note that for any $r\ge 1$,
\begin{align*}
    \sum_{|h|>0}|h|^r\mn{\Gamma_j}_{\max} 
    &= 2\sum_{h=1}^\infty h^r \left|\sum_{j=0}^\infty (C_j\otimes C_{j+h})\vec(\E[\X_1\X_1'])\right|_\infty\\
    &\le 2\mn{\E[\X_1\X_1']}_{\max}\max_{i\le n^2}\sum_{j=0}^\infty\sum_{h =1}^\infty h^r \left|\bs{e}_i(C_j\otimes C_{j+h})\right|_1\\
    &\le 2\mn{\E[\X_1\X_1']}_{\max}\max_{i\le n}\sum_{j=0}^\infty \left|\bs{e}_i'C_j\right|_1 \max_{i\le n}\,\sum_{j=0}^\infty (j+1)^{r+1}\left|\bs{e}_i'C_j\right|_1,
\end{align*}
where the last term is bounded by Assumptions DGP 1 and 4(b).

A for the estimation error,  observe that
\[
 \frac{1}{2\pi}\mn{\sum_{h=-T+1}^{T-1}\kappa_{g,\epsilon}\left(\frac{h}{M_T}\right)\left[\widehat{\Gamma}_T(h)-\Gamma_T(h)\right]}_{\max}\le \frac{1}{\pi}\sum_{h=0}^{T-1}\kappa_{g,\epsilon}\left(\frac{h}{M_T}\right)\Delta_T(h),
\]
where $\Delta_T(h)$ is defined in \eqref{eq:delta}, and for $H_T = \sum_{h=-T+1}^{T-1}\kappa_{g,\epsilon}({h}/{M_T})$
\begin{equation}\label{eq:prob_estimation}
    \Pr\left(\sum_{h=0}^{T-1}\kappa_{g,\epsilon}\left(\frac{h}{M_T}\right)\Delta_T(h)\ge x\right)\le\sum_{h=0}^{T-1}\Pr\left(\Delta_h \ge\frac{x}{H_T}\right).
\end{equation}

Now, let $\bar{c}_T=\bar{c}_{T-h}\cdot(T-h)/T$ and 
\begin{equation}
    x \gtrsim \frac{H_T}{\sqrt{T}}\log(n)^{1+\frac{2}{\alpha}+\frac{1}{\gamma}}(\log\log(n)+\log(T)+\tau)^{2/\alpha},
\end{equation}
for any $\tau>0$ and all $T$ sufficiently large.
It follows from Theorem \ref{t:acovmatrix}, and Equations \eqref{eq:estimation} and \eqref{eq:prob_estimation} that
\begin{equation}\label{eq:F_estimation_bound}
     \mn{\widehat{\bs{F}}(w)-\E\widehat{\bs{F}}(w)}_{\max} \lesssim \frac{H_T}{\sqrt{T}}\log(n)^{1+\frac{2}{\alpha}+\frac{1}{\gamma}}(\log\log(n)+\log(T)+\tau)^{2/\alpha},
\end{equation}
with probability at least
\[
1-c_1e^{-\tau} - \frac{c_2 T\log(n)^{1/\gamma}}{n^2}-\frac{c_3\max_{t\le T}\bar{c}_t\, T^{3/2}}{n^2\log(n)(\log(T)+\tau)}-4e^{-c_4T^{\alpha/4}\log(n)^{\alpha/4+\alpha/2\gamma}(\log(T)+\tau)},
\]
for $c_1,...,c_4$ not depending on $T$ or $n$.

Finally, combining Equation \eqref{eq:bias-estimation} with Equations \eqref{eq:F_bias_bound} and \eqref{eq:F_estimation_bound}, for any $w\in[0,2\pi]$
\begin{equation}\label{eq:F_error_bound}
\begin{split}
    \mn{\widehat{\bs{F}}(w) -\bs{F}(w)}_{\max}\lesssim&  \frac{2c_r}{\pi\epsilon^r M_T^r} + \frac{c_0}{\pi T}\\& +\frac{H_T}{\sqrt{T}}\log(n)^{1+\frac{2}{\alpha}+\frac{1}{\gamma}}(\log\log(n)+\log(T)+\tau)^{2/\alpha}
\end{split}
\end{equation}
\end{proof}
\begin{remark}
    Frequently, the HAC estimator is calculated from an estimator $\widehat\Y_t$ of $\Y_t$. For example, let $\Y_t = f(\Z_t;\theta_0)$ where the population parameter $\theta_0$ is estimated by $\widehat\theta_T$. In this case, we have access to $\widehat\Y_t = f(\Z_t;\widehat\theta_T)$ and calculate
    \[
    \widetilde\Gamma_T(h) = \frac{1}{T}\sum_{t=h+1}^T\widehat\Y_t\widehat\Y_{t-h}'\quad\mbox{and}\quad \widetilde\Gamma_T(-h) = \widetilde\Gamma_T(h),\quad h=0,1\ldots,T-1.
    \]
    The spectral density estimator is
    \[
    \widetilde{\bs{F}}(w) = \sum_{h=-T+1}^{T-1}\kappa_{g,\epsilon}(h/M_T)\widetilde\Gamma_T(h).
    \]
    { 
    Our goal is to obtain a non-asymptotic bound 
    \[
    \mn{\widetilde{\bs{F}}(w)-\bs{F}(w)}_{\max}\le \mn{\widehat{\bs{F}}(w)-\bs{F}(w)}_{\max} + \mn{\widetilde{\bs{F}}(w)-\widehat{\bs{F}}(w)}_{\max} \le \xi_{n,T}^{(1)}+\xi_{n,T}^{(2)},
    \]
    that holds with high probability. We have already shown that, with high probability, $\mn{\widehat{\bs{F}}(w)-\bs{F}(w)}_{\max}\le \xi_{n,T}^{(1)}$, given by the right-hand side of Equation \eqref{eq:F_error_bound}. Hence, we have to bound $\mn{\widetilde{\bs{F}}(w)-\widehat{\bs{F}}(w)}_{\max}\le\xi_{n,T}^{(2)}$. After simple algebra
    \begin{align*}
    \widetilde\Gamma_T(h)-\widehat\Gamma_T(h)
    &= \frac{1}{T}\sum_{t=h+1}^T\left(\Y_t\Y_{y-h}' - \widehat\Y_t\widehat\Y_{t-h}'\right)\\
    &= \frac{1}{T}\sum_{t=h+1}^T\left(\Y_t-\widehat\Y_t\right)\left(\Y_{y-h}-\widehat\Y_{t-h}\right)'\\
    &\quad +\frac{1}{T}\sum_{t=h+1}^T\left(\Y_t-\widehat\Y_t\right)\Y_{y-h}'\\
    &\quad + \frac{1}{T}\sum_{t=h+1}^T\Y_t\left(\Y_{y-h}-\widehat\Y_{t-h}\right)'.
    \end{align*}
    Therefore, by the triangle inequality and Hölder inequality,
    \begin{align*}
    \mn{\widetilde\Gamma_T(h)-\widehat\Gamma_T(h)}_{\max} 
    &\le \max_{i\le n}\frac{1}{T}\sum_{t=1}^T\left(Y_{i,t}-\widehat Y_{i,t}\right)^2\\
    &\quad +2\sqrt{\max_{i\le n}\frac{1}{T}\sum_{t=1}^T\left(Y_{i,t}-\widehat Y_{i,t}\right)^2}\cdot\sqrt{\max_{i\le n}\frac{1}{T}\sum_{t=1}^{T}Y_{i,t}^2}\\
    &=(1+o_p(1))\max_{i\le n}\sqrt{\frac{1}{T}\sum_{t=1}^T\left(Y_{i,t}-\widehat Y_{i,t}\right)^2}
    \end{align*}
}
    
    Finally, if the estimation error is bounded, that is, $\max_{i\le n}\sqrt{\frac{1}{T}\sum_{t=1}^T\left(Y_{i,t}-\widehat Y_{i,t}\right)^2} < \delta_{n,t}$ with probability $1-\eta_{n,T}$, we may take $\xi_{n,T}^{(2)}\propto \delta_{n,T}$ to obtain
    \[
    \xi_{n,T} \propto \frac{2c_r}{\pi\epsilon^r M_T^r} + \frac{c_0}{\pi T} +\frac{H_T}{\sqrt{T}}\log(n)^{1+\frac{2}{\alpha}+\frac{1}{\gamma}}(\log\log(n)+\log(T)+\tau)^{2/\alpha} + \delta_{n,T}.
    \]
    
\end{remark}
\end{document}